\documentclass[12pt,reqno]{amsart}

\usepackage[margin=1in]{geometry}
\usepackage[fleqn,tbtags]{mathtools}
\usepackage[shortlabels]{enumitem}
\usepackage{graphicx}
\usepackage{amssymb}
\usepackage{amsmath}
\usepackage{amsthm}
\IfFileExists{mlmodern.sty}{\usepackage{mlmodern}}{\usepackage{lmodern}}
\usepackage{eucal}
\usepackage{microtype}
\usepackage[nodisplayskipstretch]{setspace}
\usepackage{tikz}
\usepackage{tikz-cd}
\usetikzlibrary{arrows,positioning,calc}
\usepackage{xcolor}
\usepackage[colorlinks=true,linkcolor=red,citecolor=blue,urlcolor=red,
hypertexnames=false]{hyperref}
\usepackage[nameinlink, capitalize, noabbrev]{cleveref}

\numberwithin{equation}{section}
\newtheorem{theorem}{Theorem}[section]
\newtheorem{proposition}[theorem]{Proposition}
\newtheorem{lemma}[theorem]{Lemma}

\newtheorem{conjecture}[theorem]{Conjecture}
\newtheorem{definition}[theorem]{Definition}
\theoremstyle{remark}
\newtheorem{remark}[theorem]{Remark}
\newcommand{\ddc}{dd^c}
\newcommand{\PSH}{\operatorname{PSH}}

\newcommand{\vol}{\operatorname{vol}}
\newcommand{\BC}{\mathrm{BC}}
\newcommand{\PH}{\mathcal{PH}}
\newcommand{\OO}{\mathcal O}
\newcommand{\RR}{\mathbb R}
\newcommand{\CC}{\mathbb C}
\newcommand{\np}[1]{\left\langle #1\right\rangle}
\newcommand{\norm}[1]{\left\lVert #1\right\rVert}

\title[Plurigeneras of Kähler families]{Psh variation of singular Kähler-Einstein metrics \\ and \\ Invariance of plurigenera for smooth K\"ahler families}

\author[Y. Li]{Yinji Li}
\address{Yinji Li: Institute of Mathematics\\Academy of Mathematics and Systems Science\\Chinese Academy of
	Sciences\\Beijing\\100190\\P. R. China}
\email{liyinji@amss.ac.cn}
\author[M. P\u aun]{Mihai P\u aun}
\address{Mihai P\u aun: Institut f\"ur Mathematik, Universit\"at Bayreuth, 95440 Bayreuth, Germany}
\email{mihai.paun@uni-bayreuth.de}
\author[Z. Wang]{Zhiwei Wang}
\address{Zhiwei Wang: Laboratory of Mathematics and Complex Systems (Ministry of Education)\\ School of Mathematical Sciences\\ Beijing Normal University\\ Beijing 100875\\ P. R. China}
\email{zhiwei@bnu.edu.cn}

\author[X. Zhou]{Xiangyu Zhou}
\address{Xiangyu Zhou: Institute of Mathematics\\ Academy of Mathematics and Systems Science\\
	and Hua Loo--Keng Key Laboratory of Mathematics\\ Chinese Academy of Sciences\\ Beijing 100190\\ P. R. China}
\email{xyzhou@math.ac.cn}



\begin{document}
	
	\begin{abstract}
		Let $p:X\rightarrow\Delta$ be a proper holomorphic submersion whose total space $X$ is a 
		K\"ahler manifold. We prove that $\dim_{\mathbb C}H^0(X_t,mK_{X_t})$ is independent of
		$t$ for every integer $m\geq1$. This was conjectured by Y.-T. Siu in 2002.
	\end{abstract}
	
	\maketitle
	
	\tableofcontents
	
	\section{Introduction}\label{sec:intro}
	
	This paper proves the deformation invariance of plurigenera for smooth
	K\"ahler families. Let $p:X\to\Delta$ be a proper holomorphic submersion
	onto the unit disk, and assume that the total space $X$ is K\"ahler.
	Siu proposed the following conjecture.
	\begin{conjecture}[{\cite[Conjecture 2.1]{Siu022}}]\label{conj: Siu}
		For every integer $m\geq1$, the number
		\[
		P_m(X_t):=\dim_{\mathbb C}H^0(X_t,mK_{X_t})
		\]
		is independent of $t\in\Delta$.
	\end{conjecture}
	
	Siu first proved invariance for smooth projective families of manifolds
	of general type \cite{Siu98}; Kawamata subsequently gave an algebraic
	proof \cite{Kaw99}. Siu established the general projective case in
	\cite{Siu021}, and P\u aun later gave a one-tower proof \cite{Pau07}.
	Takayama extended the result to algebraic families over a smooth curve
	with canonical singularities \cite[Theorem 1.1]{Tak07}. These extension
	results have played an important role in the development of modern
	birational geometry.
	
	In the K\"ahler setting, the first results are due to Levine
	\cite{L83,L85}. His 1983 theorem assumes the existence of a nonzero
	pluricanonical section with smooth divisor; the 1985 theorem allows
	certain mild singularities of the associated cyclic cover. Cao--P\u aun
	developed this approach in \cite{CP20}, establishing an infinitesimal
	extension theorem that we briefly recall.
	For integers $m>1$ and $k\geq1$, set
	\[
	\mathcal F_k=mK_X\otimes
	\bigl(\mathcal O_X/t^{k+1}\mathcal O_X\bigr),
	\]
	where $t$ also denotes the pullback of the standard coordinate on
	$\Delta$. Let $s\in H^0(X,\mathcal F_k)$, and let $s_0$ be its
	restriction to $X_0$. Write $h_L=e^{-\varphi_L}$ for the singular
	Hermitian metric on $L_0=(m-1)K_{X_0}$ induced by $s_0$.
	Suppose that $s$ has a smooth representative $\widetilde s_k$ near
	$X_0$ such that
	\[
	\bar\partial\widetilde s_k=t^{k+1}\Lambda_k,
	\qquad
	\gamma_k=\left.\frac{\Lambda_k}{dt}\right|_{X_0},
	\]
	and
	\begin{equation}\label{eq:integrability}
		\int_{X_0}|\gamma_k|_\omega^2
		e^{-(1-\varepsilon)\varphi_L-\varepsilon\phi_L}
		\,dV_\omega<\infty
		\qquad\text{for every }0<\varepsilon<1.
	\end{equation}
	Here $\omega$ is a K\"ahler metric on $X_0$, $e^{-\phi_L}$ is a smooth
	reference metric on $L_0$, and division by $dt$ identifies the
	restriction of $\Lambda_k$ with an $L_0$-valued $(n,1)$-form on $X_0$.
	Then there exists $\widehat s\in H^0(X,\mathcal F_{k+1})$ such that
	\[
	\pi_k(\widehat s)=s,
	\qquad \pi_k:\mathcal F_{k+1}\longrightarrow\mathcal F_k.
	\]
	Thus, under \eqref{eq:integrability}, $s$ extends from the $k$th
	infinitesimal neighborhood of $X_0$ to the $(k+1)$st. This is an
	infinitesimal version of the Ohsawa--Takegoshi extension theorem.
	
	More recently, P\u aun \cite{PC} and Deng--Ning--Wang--Zhou
	\cite{DNWZ26} proved Conjecture~\ref{conj: Siu} under the assumption
	that $K_{X_t}$ is nef for every $t\in\Delta$. The present work builds
	on the metric construction used in the nef case and removes this
	additional assumption.\\
	
	Our main result is the following.
	\begin{theorem}[=Theorem~\ref{thm: inv-Kah}]\label{thm:plurigenera}
		For every integer $m\geq1$, every $t_0\in\Delta$, and
		every $s\in H^0(X_{t_0},mK_{X_{t_0}})$, there exists
		$\widetilde s\in H^0(X,mK_X)$ whose restriction to $X_{t_0}$, under the
		identification induced by the standard coordinate on $\Delta$, is $s$.
		Consequently, Conjecture~\ref{conj: Siu} holds.
	\end{theorem}
	
	The proof combines four ingredients:
	\begin{itemize}
		\item General $L^2$ extension theorems for weakly pseudoconvex
		K\"ahler manifolds and singular metrics \cite{ZZ17,ZZ20}, based on the solution of Demailly's strong openness conjecture
		\cite{GZ15}.
		\item The recent progress in the K\"ahler minimal model program
		(MMP) due to Hacon--Xie \cite{HX}, which provides a birational
		description of the big adjoint classes and constancy of their
		fiberwise volumes.
		\item Uniform estimates for degenerate complex Monge--Amp\`ere
		equations in families, including classes whose volumes may collapse,
		due to Di Nezza--Guedj--Guenancia \cite{DNGG} and Guedj--Lu
		\cite{GL25}, and the $L^\infty$ estimates of Li--Wang--Zhou
		\cite{LWZ}.
		\item An entropy comparison based on local Bergman kernels. This is 
		the most important part of the proof, extending the positivity
		result of \cite{Pau17} for the fiberwise twisted
		K\"ahler--Einstein metrics in big classes. One of the important ideas here is 
		taken from the proof of the convexity of Mabuchi energy of 
		Berman-Berndtsson, see \cite{BB}. 
	\end{itemize}
	
	We now explain how these ingredients fit together. Fix a nonzero
	section $s\in H^0(X_0,mK_{X_0})$. The $L^2$ extension theorem reduces
	the problem to constructing a singular Hermitian metric $h$ on
	$K_{X/\Delta}$ with semipositive curvature such that $s$ is locally
	square integrable with respect to $h^{\otimes(m-1)}|_{X_0}$. The
	standard coordinate on $\Delta$ identifies $K_{X/\Delta}$ with
	$K_X$. As in \cite{PC,DNWZ26}, we construct a metric independent of
	$s$ whose restriction to the central fiber has minimal singularities.
	
	Choose K\"ahler forms $\omega$ and $\beta$ on $X$, and let
	$h_\omega$ be the induced metric on $K_{X/\Delta}$, with curvature
	$\theta=\theta_\omega$. We seek, after shrinking the disk, a function
	$\Phi\in\PSH(X,\theta)$ such that $\Phi|_{X_0}$ has minimal
	singularities. The resulting metric is $h=h_\omega e^{-\Phi}$.
	Since $s\neq0$, the canonical bundle $K_{X_0}$ is pseudoeffective.
	The result of Cao--P\u aun \cite{CP25} then implies that $K_{X_t}$
	is pseudoeffective for every $t\in\Delta$. Hence the classes
	\[
	[\alpha_{\eta,t}]:=[\theta_t+\eta\beta_t]
	\]
	are big for every $\eta>0$. On each fiber we solve
	\[
	\np{(\theta_t+\eta\beta_t+\ddc\phi_{\eta,t})^n}
	=e^{\phi_{\eta,t}}\beta_t^n,
	\]
	where the brackets denote the non-pluripolar product. Our convention is
	$d^c=\frac{i}{4\pi}(\bar\partial-\partial)$, so that
	$\ddc=\frac{i}{2\pi}\partial\bar\partial$.
	The central analytic task is to prove positivity on the total space
	for the function $\phi_\eta(x)=\phi_{\eta,p(x)}(x)$.
	In contrast to the nef case, these potentials may be singular, and
	their fiberwise regularity does not justify differentiating the
	equation in the base direction.
	
	The K\"ahler MMP supplies the geometric reduction needed to address
	this difficulty. For each fixed $\eta>0$, it gives constancy of
	$\vol([\alpha_{\eta,t}])$ on $\Delta$, together with a proper
	bimeromorphic morphism $\mu:Z_\eta\to X$ and a semipositive form
	$P_\eta$ on $Z_\eta$. Away from a finite set of parameters, the
	induced morphisms $\mu_t:Z_{\eta,t}\to X_t$ satisfy
	\[
	[\mu_t^*\alpha_{\eta,t}]=[P_{\eta,t}]+[E_{\eta,t}].
	\]
	Here $P_{\eta,t}$ is semipositive and big, and the two terms on the
	right are the positive and negative parts of the divisorial Zariski
	decomposition in the sense of \cite{Boucksom}. In particular,
	\[
	\vol([\alpha_{\eta,t}])
	=\vol([\mu_t^*\alpha_{\eta,t}])
	=\vol([P_{\eta,t}]).
	\]
	This reduction does not by itself establish positivity of the
	fiberwise solutions on the total space. Moreover, the model and its
	decomposition depend on $\eta$, whereas the canonical metric must
	ultimately be obtained by letting $\eta\to0$.
	
	To prove the required variation, we use an entropy argument inspired
	by Berman--Berndtsson's proof of convexity of the Mabuchi functional
	\cite{BB}. In their treatment of the entropy term, local weighted
	Bergman kernels approximate Monge--Amp\`ere densities. Positivity
	of these kernels, combined with the homogeneous Monge--Amp\`ere
	equation, yields the relevant curvature inequality. We apply this
	mechanism to a Perron envelope whose boundary values are the
	fiberwise twisted K\"ahler--Einstein potentials.
	
	Fix $\eta$ and work on the bimeromorphic model. At the regularized
	stage, let $P$ denote the regularized positive form and $U$ the Perron
	envelope. Set
	\[
	T=P+\ddc U\geq0,\qquad T^{n+1}=0.
	\]
	Writing the regularized fiberwise equation as
	$(P_t+\ddc u_t)^n=e^{u_t}f_t\eta_t^n$, where $\eta_t$ is a reference
	K\"ahler form and $f_t$ the corresponding density, we compare
	\[
	\mu_t=T_t^n,\qquad \lambda_t=e^{U_t}f_t\eta_t^n.
	\]
	The desired inequality is $\mu_t\leq\lambda_t$. It is equivalent
	to the vanishing of the nonnegative entropy expression
	\begin{equation}\label{eq:positiveentropy}
		H(t)=\int_{Z_{\eta,t}}\log^+\!\left(
		\frac{d\mu_t}{d\lambda_t}\right)d\mu_t,
	\end{equation}
	where $\log^+a=\max\{\log a,0\}$. Regularization makes the density
	ratio well defined and allows us to pass the resulting measure
	inequalities to the limit.
	
	Local Bergman kernels approximate the logarithmic density ratio and
	give inequality for its positive part in the sense of currents. The divisorial
	terms on the model introduce an error. Integration against the Green
	function of the base disk bounds $H(t)$ by a boundary contribution
	and a divisorial contribution. The first tends to zero because the
	boundary values solve the regularized fiberwise equation. The second
	is controlled by intersection numbers that tend to zero with the
	regularization parameter, using the orthogonality property of the
	Fujita decomposition: \emph{imparable!}
	
	In the limit, the Perron envelope satisfies the required
	Monge--Amp\`ere measure inequality. The comparison principle then
	shows that it dominates the fiberwise solution. Since the envelope
	and the solutions agree on the boundary, its submean inequality
	passes to the solutions. This proves their plurisubharmonic
	variation without differentiating them in the base direction.
	
	The construction of the canonical metric now follows the strategy
	of the nef case. The Guedj--Lu estimate compares each potential,
	after subtraction of its fiberwise supremum, with the extremal
	envelope in $c_1(K_{X_t})+\eta[\beta_t]$. The comparison constant is
	uniform in $t$ and $\eta$, even when the class volume tends to zero.
	To preserve total-space positivity when taking a limit, we normalize
	instead by the single constant
	\[
	\log\!\left(\frac{v_\eta}{W}\right),\qquad
	v_\eta=\vol([\alpha_{\eta,t}]),\quad
	W=\int_{X_t}\beta_t^n.
	\]
	Both $v_\eta$ and $W$ are independent of $t$. The normalized
	potentials have uniformly bounded fiberwise suprema, so the envelope
	comparison survives passage to a quasi-psh limit $\Phi$. Its
	restriction to $X_0$ has minimal singularities in $c_1(K_{X_0})$.
	Consequently, for $s\in H^0(X_0,mK_{X_0})$,
	\[
	|s|^2h_\omega^{m-1}e^{-(m-1)\Phi}
	\leq C|s|^{2/m},
	\]
	up to smooth reference factors. The right-hand side is integrable,
	and the $L^2$ extension theorem \cite{ZZ17,ZZ20} applies.
	
	The variation result used in this argument is of independent interest.
	\begin{theorem}\label{thm:variation}
		Let $p:X\to\Delta$ be as above, and let $\omega$ and $\beta$ be
		K\"ahler forms on $X$. Set
		\[
		\alpha=\theta_\omega+\beta,
		\]
		where $\theta_\omega$ is the curvature of the metric on
		$K_{X/\Delta}$ induced by $\omega$. Assume that $[\alpha_t]$ is
		big for every $t\in\Delta$. Then
		$t\mapsto\vol([\alpha_t])$ is constant. Let
		$\varphi_t\in\PSH(X_t,\alpha_t)$ be the unique solution with full
		Monge--Amp\`ere mass of
		\begin{equation}\label{eq:original}
			\np{(\alpha_t+\ddc\varphi_t)^n}
			=e^{\varphi_t}\beta_t^n.
		\end{equation}
		The function $\varphi(x)=\varphi_{p(x)}(x)$ then satisfies
		\[
		\varphi\in\PSH(X,\alpha).
		\]
	\end{theorem}
	
	The paper is organized as follows.
	Section~\ref{thosewerethedays} reviews the construction of universal
	canonical metrics in projective families and in the nef case, and
	explains the difficulties encountered in big classes.
	Section~\ref{sec:general L2} recalls the $L^2$ extension theorems
	and the reduction to constructing a suitable canonical metric.
	Section~\ref{sec:geometry} uses the K\"ahler MMP to prove volume
	constancy and construct the bimeromorphic models.
	Section~\ref{sec:uniform estimate of MA} establishes the fiberwise
	Monge--Amp\`ere estimates.
	Sections~\ref{sec:Perron estimate} and~\ref{sec:pf of psh variation}
	develop the Perron-envelope argument and prove
	Theorem~\ref{thm:variation}.
	Finally, Section~\ref{sec:canonicalmetric} constructs a singular
	metric on $K_{X/\Delta}$ with fiberwise minimal singularities, and
	Section~\ref{sec:pluricanonicalextension} proves the extension
	theorem and invariance of plurigenera.

	\section{\texorpdfstring{Universal canonical metrics: the construction for projective families}{The projective case}}\label{thosewerethedays}

	A useful perspective on our approach is provided by two constructions
	of positively curved metrics on the relative canonical bundle with
	controlled fiberwise singularities. The first uses pluricanonical
	Bergman kernels in projective families; the second uses twisted
	K\"ahler--Einstein metrics when the fiberwise canonical bundles are nef.
	Both produce a metric independent of the pluricanonical section to be
	extended.
	\smallskip
	
	\subsection{The projective construction.}
	
	Let $p:X\to\Delta$ be a smooth projective family, and assume that
	$H^0(X_0,m_0K_{X_0})\neq0$ for some $m_0\geq1$. The construction
	explained in the notes \cite{KPV}, following
	Tsuji and Siu, produces a singular Hermitian metric on $K_{X/\Delta}$
	with semipositive curvature whose restriction to every fiber has minimal
	singularities. We call such a metric \emph{universal}.
	
	Choose a fixed auxiliary ample line bundle $A$, sufficiently positive
	for uniform global generation and $L^2$ extension. For each $m$, the
	fiberwise $m$-Bergman construction defines a metric on $mK_{X/\Delta}+A$.
	For $t\in \Delta$ over which sections extend locally, its local weight is
	\[
	\Phi_m(x)=\log\sup\left\{|s(x)|^2:
	s\in H^0(X_t,mK_{X_t}+A_t),\quad
	\int_{X_t}|s|^{2/m}e^{-\phi_A/m}\leq1\right\}.
	\]
	The integrand is intrinsically a volume form. Positivity of
	pluricanonical Bergman metrics extends this weight across the remaining
	parameters. The local upper bounds $\Phi_m\leq mC$ are uniform in $m$,
	so normalized weights admit a plurisubharmonic upper-limit construction
	as $m\to\infty$, while the auxiliary twist $A/m$ disappears. More
	precisely, after fixing smooth reference metrics, the normalized
	quasi-psh weights have curvature bounded below by
	$-m^{-1}\Theta_A$ relative to the canonical reference curvature.
	
	The crucial point is to identify the singularities of this limit on
	\emph{every} fiber. Fix $t_0$ and a metric with minimal singularities,
	of weight $\phi_{\min,t_0}$, on $K_{X_{t_0}}$. Consider
	\[
	V_{m,t_0}=H^0\!\left(X_{t_0},
	(mK_{X_{t_0}}+A_{t_0})\otimes\mathcal I(m\phi_{\min,t_0})\right).
	\]
	Uniform global generation and successive $L^2$ extension allow these
	sections to extend to the family. Suitable normalization gives
	fiberwise $L^{2/m}$ bounds independent of $m$. Hence the $m$-Bergman
	weight dominates their normalized logarithmic weights. Conversely,
	approximation of plurisubharmonic weights, using local
	Ohsawa--Takegoshi extension and global generation, shows that these
	section weights dominate $\phi_{\min,t_0}$ up to a uniform additive
	constant.
	
	The main point is that this comparison also holds at special $t\in \Delta$. For a section
	$s$ of $mK_{X_{t_0}}+ A$ extending to the total space, upper semicontinuity gives
	\[
	|s(x)|^2e^{-\Phi_m(x)}\leq
	\left(\int_{X_{t_0}}|s|^{2/m}e^{-\phi_A/m}\right)^m,
	\qquad x\in X_{t_0}.
	\]
	The limiting canonical weight therefore satisfies
	$\Phi_\infty|_{X_{t_0}}\geq\phi_{\min,t_0}-C_{t_0}$.
	Its restriction is itself a positively curved canonical weight, so
	the reverse comparison follows from the definition of minimal
	singularities. Since $t_0$ was arbitrary, the metric has minimal
	singularities on every fiber. This is stronger than controlling only
	the metrics arising from individual pluricanonical sections.
	\subsection{The nef case.}
	For smooth K\"ahler families whose fibers have nef canonical bundles,
	twisted K\"ahler--Einstein metrics replace the auxiliary ample bundle
	and pluricanonical Bergman kernels. Let $\omega$ be a K\"ahler form on
	$X$ and $\theta$ the curvature of its induced metric on $K_{X/\Delta}$.
	For $\varepsilon>0$, solve
	\[
	(\theta_t+\varepsilon\omega_t+\ddc\varphi_{\varepsilon,t})^n
	=e^{\varphi_{\varepsilon,t}}\omega_t^n.
	\]
	Nefness makes the perturbed fiber classes K\"ahler. For fixed
	$\varepsilon>0$, the solutions depend smoothly on $t$. Set
	\[
	\rho_\varepsilon=\theta+\varepsilon\omega+\ddc\varphi_\varepsilon.
	\]
	Here we use the convention $\ddc=\sqrt{-1}\partial\bar\partial$.
	In local coordinates $(z^1,\ldots,z^n,t)$ adapted to $p$, write
	$\rho_\varepsilon=\sqrt{-1}g_{A\bar B}\,dw^A\wedge d\bar w^B$,
	where $(g_{i\bar j})_{1\leq i,j\leq n}$ is the positive definite
	fiber metric and $(g^{\bar j i})$ its inverse. Its horizontal lift and
	geodesic curvature are
	\begin{align}
		v_\varepsilon
		&=\frac{\partial}{\partial t}
		-\sum_{i,j=1}^n g^{\bar j i}g_{t\bar j}
		\frac{\partial}{\partial z^i},
		\label{eq:nef-horizontal}\\
		c_\varepsilon
		&=g_{t\bar t}
		-\sum_{i,j=1}^n g_{t\bar j}g^{\bar j i}g_{i\bar t}
		=\rho_\varepsilon(v_\varepsilon,\overline{v_\varepsilon}).
		\label{eq:nef-geodesic-curvature}
	\end{align}
	Evaluation of a $(1,1)$-form on $(v,\bar v)$ here denotes its
	Hermitian value, with the factor $\sqrt{-1}$ removed.
	Differentiating the Monge--Amp\`ere equation gives the identity \cite[Section~3.2, equation~(35)]{Pau17}:
	\begin{equation}\label{eq:nef-curvature-identity}
		\left(-\Delta_{\varepsilon,t}+1\right)c_\varepsilon
		=\left|\bar\partial_{X_t}v_\varepsilon\right|^2_{\rho_{\varepsilon,t}}
		+\varepsilon\omega(v_\varepsilon,\overline{v_\varepsilon}).
	\end{equation}
	Our sign convention is
	$\Delta_{\varepsilon,t}f=\sum_{i,j}g^{\bar j i}
	\partial_i\partial_{\bar j}f$, so that $-\Delta_{\varepsilon,t}$
	is nonnegative. The tensor $\bar\partial_{X_t}v_\varepsilon$ is
	a fiberwise $(0,1)$-form with values in $T^{1,0}X_t$.
	The right-hand side of \eqref{eq:nef-curvature-identity} is
	nonnegative, and its second term is strictly positive since
	$\omega$ is K\"ahler and $dp(v_\varepsilon)=\partial/\partial t$.
	The maximum principle on the compact fiber therefore gives
	$c_\varepsilon>0$. Together with fiberwise positivity, the
	Schur-complement formula \eqref{eq:nef-geodesic-curvature} proves
	total-space positivity of $\rho_\varepsilon$.
	
	In order to take the limit as $\varepsilon\to 0$, the normalization is essential. Put
	\[
	V_\varepsilon=\int_{X_t}(\theta_t+\varepsilon\omega_t)^n,
	\quad W=\int_{X_t}\omega_t^n,
	\quad u_\varepsilon=\varphi_\varepsilon-\log(V_\varepsilon/W).
	\]
	Both integrals are independent of $t$, so the normalization subtracts
	a single constant on the total space and preserves curvature. Moreover,
	\[
	\int_{X_t}e^{u_{\varepsilon,t}}\,d\nu_t=1,
	\qquad d\nu_t=W^{-1}\omega_t^n.
	\]
	Uniform $L^1$ bounds for normalized quasi-psh functions imply
	$\int u_{\varepsilon,t}\,d\nu_t
	\geq\sup_{X_t}u_{\varepsilon,t}-C$.
	Jensen's inequality then gives
	$0\leq\sup_{X_t}u_{\varepsilon,t}\leq C$ locally uniformly in $t$
	and $\varepsilon$. This removes the divergent additive constant
	when $V_\varepsilon\to0$.
	
	The remaining input is a volume-independent comparison with the
	extremal envelope in the perturbed class. It yields, on each fiber,
	\[
	u_{\varepsilon,t}\geq
	V_{\theta_t+\varepsilon\omega_t}-C_t
	\geq V_{\theta_t}-C_t,
	\]
	where $C_t$ is independent of $\varepsilon$ and
	$V_\gamma=\sup\{v\in\PSH(X_t,\gamma):v\leq0\}$.
	For $\varepsilon_j\downarrow0$, set
	\[
	U_N=\left(\sup_{j\geq N}u_{\varepsilon_j}\right)^*,
	\qquad \Psi=\lim_{N\to\infty}U_N.
	\]
	Here regularization is taken on the total space. Each $U_N$ is
	$(\theta+\varepsilon_N\omega)$-psh; the decreasing limit is
	$\theta$-psh. The fiberwise lower bound above 
	shows that $\Psi_t$ has minimal singularities on every fiber.
	
	\subsection{Main difficulties for big families.}
	Even granting the required conclusions of the relative K\"ahler MMP,
	a substantial analytic difficulty remains. The birational reduction
	describes the positive and divisorial parts of the adjoint classes and
	supplies the volume identities needed for the argument. It does not
	provide regularity in the base direction for the fiberwise
	Monge--Amp\`ere solutions. On a suitable model they satisfy
	\[
	(P_t+\ddc u_t)^n=e^{u_t}f_t\eta_t^n,
	\]
	where $P_t$ is semipositive and big, but may degenerate, and $f_t$ may
	vanish along divisors. Boundedness of the potentials and smoothness
	on suitable open subsets of individual fibers do not imply estimates
	for transverse derivatives. Stability in pluripotential theory can
	provide continuity under appropriate hypotheses, but does not by
	itself justify differentiating the equation with respect to the base.
	
	This is where the direct curvature argument used in the nef case
	becomes difficult to apply. There, positively twisted fiberwise
	metrics are smooth and nondegenerate, and their smooth dependence
	on the parameter permits the differential identities underlying
	total-space positivity. In the big case, the analogous argument
	would require control of $\partial_tu$, $\partial_t\bar\partial_tu$
	and mixed derivatives, together with the inverse of a degenerating
	fiber metric. Even when these expressions can be computed on an
	open locus where the family of solutions is smooth, extending the
	inequality across the degeneracy locus requires additional estimates.
	Neither fiberwise regularity nor the MMP supplies those estimates.
	
	The Perron-envelope method used here avoids \emph{precisely} this transverse differentiation
	of the fiberwise solutions.

	\section{\texorpdfstring{General $L^2$ extension theorem}{General L2 extension theorem}}
	\label{sec:general L2}
	In a series of work \cite{GZ15SC,GZ15L2,GZ17,ZZ17,ZZ20,ZZ20PJM}, general $L^2$ extension theorems with optimal estimates have being established both in the settings of almost Stein and weakly pseudoconvex K\"ahler manifolds for sections valued in bundles with singular metrics. 
	In this article we use the following version of $L^2$ extension theorem, which essentially uses strong openness theorem \cite{GZ15} and generalized Siu's lemma by Zhou-Zhu.
	\begin{theorem}[{\cite[Theorem 1.1 and Remark 1.3]{ZZ17}}]\label{thm: ext}
		Let $(X,\omega)$ be a weakly pseudoconvex K\"ahler manifold of
		dimension $n$, and let $L$ be a holomorphic line bundle over $X$ with a
		singular Hermitian metric $h$. Let $s:X\rightarrow\mathbb C^r$ be a
		holomorphic map, $1\leq r\leq n$, such that $0$ is a regular value of
		$s$. Assume that $\sqrt{-1}\Theta_{L,h}\geq0$ and that
		$|s(x)|\leq M$ on $X$. Put $Y:=s^{-1}(0)$ and endow
		$\mathbb C^r$ with its standard flat Hermitian metric. The norms of
		$\bigwedge^r(ds)$, $f$, and $F$ below are induced by this metric,
		$\omega$, and $h$, and
		\[
		dV_{X,\omega}:=\frac{\omega^n}{n!},
		\qquad
		dV_{Y,\omega}:=\frac{(\omega|_Y)^{n-r}}{(n-r)!}.
		\]
		Then every
		\[
		f\in H^0\!\left(Y,(K_X\otimes L)|_Y\right)
		\]
		satisfying
		\[
		\int_Y\frac{|f|_L^2}{|\bigwedge^r(ds)|^2}\,
		dV_{Y,\omega}<+\infty
		\]
		admits an extension $F\in H^0(X,K_X\otimes L)$ with $F|_Y=f$ and
		\[
		\int_X|F|_L^2\,dV_{X,\omega}
		\leq C_{r,M}
		\int_Y\frac{|f|_L^2}{|\bigwedge^r(ds)|^2}\,
		dV_{Y,\omega},
		\]
		where $C_{r,M}$ depends only on $r$ and $M$.
	\end{theorem}
	
	\begin{remark}
		This is a direct specialization of
		\cite[Theorem 1.1 and Remark 1.3]{ZZ17}. Let
		$E:=X\times\mathbb C^r$ carry its standard flat metric and put
		\[
		\lambda:=\frac{1}{e\max\{1,M\}},
		\qquad s_\lambda:=\lambda s.
		\]
		Then $|s_\lambda|\leq e^{-1}$. Apply the cited theorem with
		$\psi=0$, $\alpha=1$, and $R(\tau)=e^{-\tau}$. The local weights of
		$h$ are plurisubharmonic, hence locally bounded above, and
		$\Theta_E=0$. Moreover, on $X\setminus Y$,
		\[
		\sqrt{-1}\Theta_{L,h}
		+r\sqrt{-1}\partial\bar\partial\log|s_\lambda|^2\geq0,
		\qquad
		r\log|s_\lambda|^2\leq-2r,
		\]
		so all three curvature-side hypotheses of
		\cite[Theorem 1.1]{ZZ17} hold. Since $C_R=1$ and
		$e^\tau R(\tau)=1$, its estimate becomes
		\[
		\int_X|F|_L^2\,dV_{X,\omega}
		\leq\frac{(2\pi)^r}{r!}
		\int_Y\frac{|f|_L^2}
		{|\bigwedge^r(ds_\lambda)|^2}\,dV_{Y,\omega}.
		\]
		As $\bigwedge^r(ds_\lambda)=\lambda^r\bigwedge^r(ds)$, one may take
		\[
		C_{r,M}=\frac{(2\pi)^r}{r!}
		\bigl(e\max\{1,M\}\bigr)^{2r}.
		\]
		A proper-family formulation with additional multiplier-ideal conclusions
		appears in \cite[Theorem 1.3]{ZZ20}.
	\end{remark}
	In our situation, $p:X\to\Delta$ is a proper submersion, $X$ is a weakly pseudoconvex K\"ahler manifold with K\"ahler form $\omega$. 
	If we can construct $\theta=\theta_{\omega}$-psh function $\Phi$ on $X$ so that $\Phi|_{X_0}\in\operatorname{PSH}(X_0,\theta_0)$ has minimal singularities (which is done in Theorem~\ref{thm:metric}), then for every $s\in H^0(X_0,mK_{X_0})$, we take $L=(m-1)K_{X/\Delta}$ equiped with metric $h_{\omega}^{\otimes(m-1)}e^{-\Phi}$ and the extension Theorem~\ref{thm: ext} applies to produce $\hat{s}\in H^0(X,mK_X)$ so that $\hat{s}|_{X_0}=s_0\wedge dt$.
	
	\section{K\"ahler MMP}\label{sec:geometry}

	We follow \cite{DHY,HX} to recall some basic notions in K\"ahler MMP.
	Let $X$ be a complex analytic variety, that is, an irreducible reduced
	complex space. Denote by $\PH_X$ the sheaf of real parts of
	holomorphic functions, by $\mathcal A_X$ the sheaf of smooth real-valued
	functions, and by $\mathcal D_X$ the sheaf of real distributions.
	A closed real $(1,1)$-form, respectively current, with local potentials
	is a global section of $\mathcal A_X/\PH_X$, respectively
	$\mathcal D_X/\PH_X$. Set
	\[
	H^{1,1}_{\mathrm{BC}}(X,\mathbb R)=H^1(X,\PH_X).
	\]
	The class of such a form $\theta$ or a current $T$ is denoted by $[\theta]$, $[T]$ respectively.
	For a holomorphic map $f:X\to Z$ and a smooth
	representative $\theta$ of $\alpha\in H^{1,1}_{\mathrm{BC}}(X,\mathbb R)$,
	the pullback class is
	\[
	f^*\alpha=[f^*\theta]
	\in H^{1,1}_{\mathrm{BC}}(X,\mathbb R).
	\]
	For a proper morphism $f:X\to S$ of normal analytic varieties,
	\[
	\alpha\equiv_S\alpha'
	\quad\Longleftrightarrow\quad
	\alpha-\alpha'=f^*\gamma
	\quad\text{for some }\gamma\in H^{1,1}_{\mathrm{BC}}(S,\mathbb R).
	\]

	\begin{definition}[$b$-currents] 
		A real closed $b$-$(1,1)$-current $\boldsymbol\beta$ over a normal
		analytic variety $X$ is a collection of real closed $(1,1)$-currents
		$\beta_V$ on all normal proper bimeromorphic models $V\to X$ such that
		$r_*\beta_W=\beta_V$ for every proper bimeromorphic morphism
		$r:W\to V$ of normal models over $X$.
		The current $\beta_V$ is called the \emph{trace} on $V$.
		The $b$-current is called positive if every trace is positive.
		We say $\boldsymbol\beta$ \emph{descends to} $V$ if $\beta_V$ has local potentials and,
		for every higher model $r:W\to V$, the trace $\beta_W$ has local
		potentials and
		\[
		[\beta_W]=r^*[\beta_V]
		\quad\text{in }H^{1,1}_{\mathrm{BC}}(W,\mathbb R).
		\]
		For a normal variety $M$ properly bimeromorphic to $X$ and a
		common resolution $X\xleftarrow{r}W\xrightarrow{s}M$, its trace
		is defined by $\beta_M:=s_*\beta_W$.
	\end{definition}

	If $T$ is a smooth closed real $(1,1)$-form with local potentials,
	or a closed positive $(1,1)$-current with local plurisubharmonic
	potentials on $X$,
	$T$ induces a $b$-current $\overline T$ defined on every higher
	model $r:W\to X$ by
	\[
	(\overline T)_W=r^*T.
	\]
	For $T=i\partial\bar\partial v$ locally with $v$ plurisubharmonic,
	this pullback is locally
	$i\partial\bar\partial(v\circ r)$.
	For a normal variety $M$ properly bimeromorphic to $X$ and a
	common resolution $X\xleftarrow{r}W\xrightarrow{s}M$, it holds that
	\[
	(\overline T)_M=s_*r^*T
	\]

	\begin{definition}[Generalized pairs,
		{\cite[Definitions~2.7]{DHY}}]
		Let $f:X\to S$ be a proper morphism of normal analytic varieties,
		where $S$ is relatively compact. A \emph{generalized pair}
		$(X/S,B+\boldsymbol\beta)$ consists of an effective real divisor
		$B$ and a real closed $b$-$(1,1)$-current $\boldsymbol\beta$
		for which there exist a resolution $\nu:X'\to X$, a real divisor
		$B'$ with simple normal crossings, and a class
		$a_X\in H^{1,1}_{\mathrm{BC}}(X,\mathbb R)$ satisfying
		\[
		\nu_*B'=B,\qquad
		\boldsymbol\beta\text{ descends to }X',\qquad
		[\beta_{X'}]\text{ is nef over }S,
		\]
		and
		\begin{equation}\label{eq:mmp-adjoint-pullback}
			[K_{X'}+B'+\beta_{X'}]=\nu^*a_X.
		\end{equation}
		The morphism $\nu$ is called a \emph{structure morphism} or
		\emph{log resolution}. We write
		\[
		a_X=[K_X+B+\beta_X].
		\]
		The class $a_X$ is called the adjoint class, $K_X$ denotes
		the reflexive canonical sheaf $\omega_X=(\textstyle\bigwedge^{\dim X}\Omega_X^1)^{**}$.
		A generalized pair over a point is called compact and is denoted by
		$(X,B+\boldsymbol\beta)$. It is K\"ahler if $X$ is K\"ahler.
	\end{definition}

	\begin{definition}[Generalized discrepancies and singularities]
		For a prime divisor $E$ over a generalized pair
		$(X,B+\boldsymbol\beta)$, choose a log resolution $\nu:X'\to X$
		on which $E$ appears and define
		\[
		a(E,X,B+\boldsymbol\beta)=-\operatorname{mult}_E B',
		\]
		where $B'$ satisfies \eqref{eq:mmp-adjoint-pullback}
		\cite[Definition~2.9(1)]{DHY}.
		The pair is \emph{generalized Kawamata log terminal} (gklt),
		respectively \emph{generalized log canonical} (glc), if
		\[
		a(E,X,B+\boldsymbol\beta)>-1,
		\qquad\text{respectively}\qquad
		a(E,X,B+\boldsymbol\beta)\geq-1,
		\]
		for every prime divisor $E$ over $X$
		\cite[Definition~2.9(2)]{DHY}.
		It is \emph{generalized canonical} if
		$a(E,X,B+\boldsymbol\beta)\geq0$ for every exceptional prime
		divisor $E$ over $X$ \cite[Section~2]{HX}.
	\end{definition}
	
	\begin{definition}[Big, modified big, and modified K\"ahler classes, {\cite[Definition~2.8]{HX}}] 
		On a compact normal K\"ahler variety, a K\"ahler current is a real
		closed $(1,1)$-current $T$ with local potentials such that
		$T\geq\varepsilon\omega$ for some K\"ahler form $\omega$ and $\varepsilon>0$. A Bott--Chern class is \emph{big} if it contains
		a K\"ahler current.
		For a proper bimeromorphic morphism $\mu:V\to X$ from a normal
		compact K\"ahler variety to a normal compact analytic variety and
		a big class $\xi\in H^{1,1}_{\mathrm{BC}}(V,\mathbb R)$,
		the push-forward $\mu_*\xi$ is called \emph{modified big}. 
		A class $\alpha\in H^{1,1}_{\mathrm{BC}}(X,\mathbb R)$ on a normal
		compact K\"ahler variety is \emph{modified K\"ahler} if there exist
		a proper bimeromorphic morphism $\mu:V\to X$ and a K\"ahler form
		$\omega_V$ such that $\mu_*\omega_V$ has local potentials
		and $\alpha=[\mu_*\omega_V]$.
	\end{definition}
	
	\begin{remark}[Bigness under modifications,
		{\cite[Definition~2.8 and Remark~2.9]{HX}}]
		\label{rem:mmp-bigness}
		Let $\nu:V\to X$ be a proper bimeromorphic morphism between
		normal compact K\"ahler varieties.
		If $T$ is a real closed current with local potentials on $V$,
		$[T]$ is big, and $\nu_*T$ has local potentials, then
		$[\nu_*T]$ is big.
		If $V$ is smooth and $\alpha\in H^{1,1}_{\mathrm{BC}}(X,\mathbb R)$
		is big, then $\nu^*\alpha$ is big.
	\end{remark}
	
	\begin{definition}[Semiampleness, {\cite[Section~2]{HX}}]
		A class $\alpha\in H^{1,1}_{\mathrm{BC}}(X,\mathbb R)$ on a
		compact analytic variety $X$ is \emph{semiample} if there are
		a holomorphic map $h:X\to Y$ to a compact analytic variety and
		a K\"ahler class $\gamma$ on $Y$ such that $\alpha=h^*\gamma$.
	\end{definition}

	A proper holomorphic map $h:X\to Y$ is a \emph{contraction} if $ h_*\mathcal O_X=\mathcal O_Y$.
	It is called \emph{projective} if $X$ admits an $h$-ample line bundle.
	A proper morphism $h:X\to Y$ of compact complex spaces is called
	\emph{Moishezon} if there is a surjective projective morphism
	$r:W\to X$ such that $h\circ r$ is projective.
	A \emph{Moishezon contraction} is a morphism satisfying both
	conditions.

	\begin{definition}[Exceptional divisors and bimeromorphic contractions]
		For a proper surjective morphism $h:X\to Y$ of normal varieties,
		a real divisor $E$ on $X$ is \emph{$h$-exceptional} if
		\[
		\operatorname{codim}_Y h(\operatorname{Supp}E)\geq2.
		\]
		For a proper bimeromorphic map $\phi:X\dashrightarrow M$,
		choose a common resolution $X\xleftarrow{r}W\xrightarrow{s}M$.
		If $D$ is a prime divisor on $X$ and $\widetilde D$ is its
		strict transform on $W$, then $D$ is called $\phi$-exceptional if
		$\operatorname{codim}_M s(\widetilde D)\geq2$.
		Set $\phi_*D=s_*\widetilde D$. 
		A proper bimeromorphic map $\phi:X\dashrightarrow M$ between
		normal varieties is called a \emph{bimeromorphic contraction} if there
		are no $\phi^{-1}$-exceptional prime divisors {\cite[Definition 2.15(4)]{HX}}.
	\end{definition}
	A proper bimeromorphic map $\phi:X\dashrightarrow M$
	between normal varieties is a bimeromorphic contraction if and
	only if on some common resolution
	$X\xleftarrow{r}W\xrightarrow{s}M$, every $r$-exceptional
	prime divisor is $s$-exceptional.

	\begin{definition}[Non-positive and negative contractions, {\cite[Definition~2.16]{HX}}]
		Let $\phi:X\dashrightarrow M$ be a proper bimeromorphic
		contraction, let $\alpha\in H^{1,1}_{\mathrm{BC}}(X,\mathbb R)$,
		and suppose that
		\[
		\alpha_M:=s_*r^*\alpha
		\in H^{1,1}_{\mathrm{BC}}(M,\mathbb R)
		\]
		on some common resolution $X\xleftarrow{r}W\xrightarrow{s}M$.
		The map $\phi$ is called \emph{$\alpha$-non-positive}  if some such
		resolution satisfies
		\begin{equation}\label{eq:mmp-non-positive}
			r^*\alpha=s^*\alpha_M+[E],
			\qquad E\geq0,\qquad E\text{ is }s\text{-exceptional}.
		\end{equation}
		It is called \emph{$\alpha$-negative} if in addition,
		$\operatorname{Supp}E$ contains the strict transform of every
		$\phi$-exceptional prime divisor on $X$.
	\end{definition}
	
	\begin{definition}[Models of gklt pairs, {\cite[Definition~2.11 and Lemma~2.12(1)]{DHY}}]
		Let $(X,B+\boldsymbol\beta)$ be a compact gklt pair and
		$\phi:X\dashrightarrow M$ a proper bimeromorphic map.
		Put $B_M=\phi_*B+E_M$, where $E_M$ is the reduced sum of the
		$\phi^{-1}$-exceptional prime divisors. Assume that
		$(M,B_M+\boldsymbol\beta)$ is a generalized pair, and set
		$a_M=[K_M+B_M+\beta_M]$.
		$\phi:X\dashrightarrow M$ is called a \emph{weak log canonical model} if
		$(M,B_M+\boldsymbol\beta)$ is glc, $a_M$ is nef, and
		\[
		a(F,X,B+\boldsymbol\beta)
		\leq a(F,M,B_M+\boldsymbol\beta)
		\]
		for every $\phi$-exceptional prime divisor $F$ on $X$.
		It is a \emph{log terminal model} if $\phi$ is a bimeromorphic
		contraction, $(M,\phi_*B+\boldsymbol\beta)$ is strongly
		$\mathbb Q$-factorial and gklt, $a_M$ is nef, and the above
		inequality is strict for every $\phi$-exceptional prime divisor.
		A log terminal model is called \emph{good} if $a_M$ is semiample.
	\end{definition}

	\begin{definition}[Divisorial decomposition]
		For a big class $\alpha$ on a compact K\"ahler manifold $V$,
		let
		\[
		\nu(\alpha,D)
		=\inf_{\substack{T\in\alpha\\ T\geq0}}\nu(T,D),
		\qquad
		N(\alpha)=\sum_{D\subset V}\nu(\alpha,D)D,
		\]
		where $T$ ranges over closed positive currents and $D$ over
		prime divisors; $\nu(T,D)$ is the generic Lelong number along $D$.
		For a big class $\alpha$ on a normal compact K\"ahler variety
		$X$ and a K\"ahler resolution $\nu:V\to X$, set
		\[
		N(\alpha)=\nu_*N(\nu^*\alpha).
		\]
		For real divisors $A=\sum_i a_iD_i$ and $B=\sum_i b_iD_i$,
		indexed by the union of their prime components with zero
		coefficients inserted where necessary, define
		\[
		A\wedge B=\sum_i\min\{a_i,b_i\}D_i.
		\]
	\end{definition}

	\begin{proposition}
		\label{prop:mmp-discrepancies}
		Let $(X,B+\boldsymbol\beta)$ be a compact K\"ahler gklt pair.
		\begin{enumerate}
			\item If $\phi:X\dashrightarrow M$ is a weak log canonical model,
			then
			\[
			a(F,X,B+\boldsymbol\beta)
			\leq a(F,M,B_M+\boldsymbol\beta)
			\]
			for every prime divisor $F$ over $X$, and the induced pair
			on $M$ is gklt \cite[Lemma~2.12(1)]{DHY}.
			\item If $\phi:X\dashrightarrow M$ is a log terminal model and $M$ is K\"ahler,
			then $\phi$ is $a_X$-non-positive by \cite[Lemma~2.19]{HX}, which means there is a common resolution
			$X\xleftarrow{r}W\xrightarrow{s}M$ satisfying
			\[
			r^*a_X=s^*a_M+[E],
			\qquad E\geq0,\qquad E\text{ is }s\text{-exceptional}.
			\]
		\end{enumerate}
	\end{proposition}
	
	\begin{proposition}
		\label{prop:mmp-rational}
		
		\begin{enumerate}
			\item If $(X,B+\boldsymbol\beta)$ is gklt, then $X$ has rational
			singularities \cite[Lemma~2.25]{HX}. 
			\item Let $h:M\to Y$ be a proper bimeromorphic holomorphic map
			between normal compact varieties in Fujiki's class $\mathcal C$.
			Suppose that $(M,B_M+\boldsymbol\beta)$ is gklt and
			\[
			[K_M+B_M+\beta_M]=h^*\gamma,\qquad
			\gamma\in H^{1,1}_{\mathrm{BC}}(Y,\mathbb R).
			\]
			Then $(Y,h_*B_M+\boldsymbol\beta)$ is gklt,
			\[
			[K_Y+h_*B_M+\beta_Y]=\gamma,
			\]
			and $Y$ has rational singularities
			\cite[Lemmas~2.31]{HX}.
		\end{enumerate}
	\end{proposition}
	
	\begin{proposition}[{\cite[Lemma~2.39]{HX}}]
		\label{prop:mmp-bc-pullback}
		Let $h:M\to Y$ be a proper surjective holomorphic map with
		connected fibers between normal compact analytic varieties
		with rational singularities. Suppose that $M$ belongs to
		Fujiki's class $\mathcal C$ and
		\[
		R^ih_*\mathcal O_M=0\qquad(i>0).
		\]
		Then
		\[
		h^*:H^{1,1}_{\mathrm{BC}}(Y,\mathbb R)
		\longrightarrow H^{1,1}_{\mathrm{BC}}(M,\mathbb R)
		\]
		is injective. 
	\end{proposition}
	
	\begin{theorem}
		\label{prop:mmp-existence}
		Let $(X,B+\boldsymbol\beta)$ be a compact K\"ahler gklt pair,
		and put $a_X=[K_X+B+\beta_X]$.
		\begin{enumerate}
			\item If $a_X$ is big, or if $a_X$ is pseudoeffective and
			$B+\beta_X$ is big, then the pair admits a good log terminal
			model \cite[Theorem~1.1(2)]{HX}.
			\item If $X$ is strongly $\mathbb Q$-factorial,
			$B+\beta_X$ is big, and $\omega$ is a modified K\"ahler class
			such that $a_X+\omega$ is nef, then an $a_X$-MMP with scaling
			of $\omega$ terminates with a good log terminal model or a
			Mori fiber space \cite[Theorem~1.5]{HX}.
		\end{enumerate}
	\end{theorem}
	
	\begin{theorem}[Transcendental base-point-free theorem,
		{\cite[Theorem~1.4]{HX}}]
		\label{thm:mmp-basepointfree}
		Let $(M,B_M+\boldsymbol\beta)$ be a compact K\"ahler gklt pair.
		Suppose that
		\[
		a_M=[K_M+B_M+\beta_M]\text{ is nef},
		\qquad B_M+\beta_M\text{ is modified big}.
		\]
		Then there are a Moishezon contraction $h:M\to Y$ onto a
		compact K\"ahler space and a K\"ahler class
		$\gamma\in H^{1,1}_{\mathrm{BC}}(Y,\mathbb R)$ such that
		\[
		a_M=h^*\gamma.
		\]
	\end{theorem}
	
	\begin{theorem}[Extension of MMP sequences, {\cite[Theorem~7.7]{HX}}]
		\label{lem:mmp-extension}
		Let $(X/S,B+\boldsymbol{\beta})$ be a generalized pair and $g:X\to S$ be flat, proper, with connected fibers where $S$ is a connected, smooth, relatively compact curve.
		Fix $0\in S$ and write $X_0=g^{-1}(0)$, with
		$X_0\not\subset\operatorname{Supp}B$.
		Suppose the generalized pair $(X_0,B_0+\boldsymbol{\beta}_0)$ is K\"ahler and has canonical singularities, where
		\[
		\left.
		\bigl(K_X+X_0+B+\boldsymbol{\beta}_X\bigr)
		\right|_{X_0}
		=
		K_{X_0}+B_0+\boldsymbol{\beta}_0
		\]
		and that
		\[
		\lfloor B_0\rfloor=0,
		\qquad
		N\bigl(K_{X_0}+B_0+\boldsymbol{\beta}_0\bigr)
		\wedge B_0=0.
		\]
		For any sequence of transcendental MMP steps
		\[
		(X_0,B_0+\boldsymbol{\beta}_0)
		\dashrightarrow
		(X_0^{(1)},B_0^{(1)}+\boldsymbol{\beta}_0^{(1)})
		\dashrightarrow
		(X_0^{(2)},B_0^{(2)}+\boldsymbol{\beta}_0^{(2)})
		\dashrightarrow\cdots,
		\]
		there is an open neighborhood $U\subset S$ containing $0$
		on which this sequence extends to
		\[
		(X,B+\boldsymbol{\beta})/U
		\dashrightarrow
		(X^{(1)},B^{(1)}+\boldsymbol{\beta}^{(1)})/U
		\dashrightarrow
		(X^{(2)},B^{(2)}+\boldsymbol{\beta}^{(2)})/U
		\dashrightarrow\cdots,
		\]
		a sequence of proper meromorphic maps which are
		$(K_X+B+\boldsymbol{\beta}_X)$-negative.
	\end{theorem}
	
	\begin{theorem}[{\cite[Corollary~7.8]{HX}}]
		\label{cor:mmp-psef-family}
		Let $p:X\to S$ be a proper holomorphic submersion with connected
		fibers, where $X$ is a K\"ahler manifold and $S$ is a smooth
		connected relatively compact curve. For $0\in S$,
		\[
		K_{X_0}\text{ is pseudoeffective}
		\quad\Longleftrightarrow\quad
		K_{X_t}\text{ is pseudoeffective for every }t\in S.
		\]
	\end{theorem}

	With these preparations, we can prove the constancy of $\vol([\alpha_t])$ in Theorem~\ref{thm:variation} and resolve the singularity of big classes $[\alpha_t]$ by an appropriate bimeromorphic morphism $\mu:Z\to X$.
	We need the following extension result.
	\begin{lemma}\label{lem:descent}
		Let  $h_0:M_0\to Y_0$ be a bimeromorphic morphism between compact connected normal K\"ahler spaces with rational singularities.
		Let $\pi_M:M\to\Delta$ be a proper flat morphism with central fiber $M_0$ and $M$ is normal.
		Suppose
		$a\in H^{1,1}_{\BC}(M,\RR)$ satisfies $a|_{M_0}=h_0^*\gamma_0$ for a
		K\"ahler class $\gamma_0$ on $Y_0$. 
		After shrinking $\Delta$ about $0$, there exists
		a proper flat deformation $\pi_Y:Y\to\Delta$ with normal total space and normal
		fibers, a proper bimeromorphic holomorphic map $h:M\to Y$ over $\Delta$ extending
		$h_0$, and a K\"ahler class $\gamma\in H^{1,1}_{\BC}(Y,\RR)$ such that
		$h^*\gamma=a$ and $\gamma|_{Y_0}=\gamma_0$.
	\end{lemma}
	
	\begin{proof}
		We first extend $h_0:M_0\to Y_0$ to $h:M\to Y$ by a fundamental result of Koll\'ar--Mori \cite{KM} and prove that $R^ih_*\OO_M=0$ for every $i>0$. Then the isomorphism $H^i(Y,\OO_Y)\xrightarrow{\sim}H^i(M,\OO_M)$ will give the existence of the desired $\gamma\in H^{1,1}_{\mathrm{BC}}(X,\RR)$.
		
		Since $M_0,Y_0$ are normal, the bimeromorphic morphism satisfies $(h_{0})_*\OO_{M_0}=\OO_{Y_0}$. Furthermore, since they has rational singularities, take a common resolution $W\xrightarrow{p} M_0,W\xrightarrow{q} Y_0$, we have
		\begin{align*}
			&R^ip_*\mathcal{O}_W=0, \quad R^iq_*\mathcal{O}_W=0, \qquad i>0,\\
			&p_*\OO_{W}=\OO_{M_0}, \quad q_*\OO_{W}=\OO_{Y_0}.
		\end{align*}
		By relative Leray spectral sequence (see e.g. Stacks  Project, Lemma 20.13.8, Tag01F6), we get $R^i(h_0)_*\OO_{M_0}=0$ for $i>0$. 
		
		\noindent\textbf{Proposition.}
		Let $\pi_M\colon M\to\Delta$ be a proper flat morphism of
		complex spaces.
		Suppose that $h_0\colon M_0\to Y_0$ is a proper morphism such that
		\[
		(h_0)_*\mathcal O_{M_0}=\mathcal O_{Y_0},
		\qquad
		R^1(h_0)_*\mathcal O_{M_0}=0.
		\]
		Then, after shrinking $\Delta$ around $0$, there exist a proper
		flat morphism $\pi_Y\colon Y\to\Delta$ so that
		$\pi_Y^{-1}(0)\simeq Y_0$, and a proper contraction $h\colon M\to Y$ with $h_*\OO_M=\OO_Y$
		such that $h|_{M_0}=h_0$ and
		\[
		\begin{tikzcd}[column sep=0.4em, row sep=1.0em]
			M \arrow[rr, "h"] \arrow[dr, "\pi_M"']
			&& Y \arrow[dl, "\pi_Y"] \\
			& \Delta &
		\end{tikzcd}
		\]
		\begin{proof}
			This extension result of contraction morphism is a direct consequence of \cite[Proposition~11.4]{KM} and we include its proof for the convenience of readers.
			Let
			\[
			p\colon\mathcal M\to \operatorname{Def}(M_0),
			\qquad
			q\colon\mathcal Y\to \operatorname{Def}(Y_0)
			\]
			be proper flat representatives of the versal deformation families.
			By \cite[Proposition~11.4]{KM} applies to $h_0$, after choosing sufficiently small
			representatives, there exist morphisms
			\[
			H\colon \operatorname{Def}(M_0)\to \operatorname{Def}(Y_0),
			\qquad
			\mathcal H\colon\mathcal U\to\mathcal V
			\]
			such that $\mathcal H|_{M_0}=h_0$ and
			\[
			\begin{tikzcd}[row sep=large, column sep=large]
				\mathcal M
				\arrow[r, "\mathcal H"]
				\arrow[d, "p"']
				&
				\mathcal Y
				\arrow[d, "q"]
				\\
				\operatorname{Def}(M_0)
				\arrow[r, "H"']
				&
				\operatorname{Def}(Y_0).
			\end{tikzcd}
			\]
			By versality, after shrinking $\Delta$, the given deformation
			$\pi_M$ is induced by a morphism
			$a\colon(\Delta,0)\to(\operatorname{Def}(M_0),0)$:
			\[
			M\simeq\mathcal M\times_{\operatorname{Def}(M_0)}\Delta.
			\]
			Define
			\[
			Y:=\mathcal Y\times_{\operatorname{Def}(Y_0),H\circ a}\Delta,
			\]
			and let $\pi_Y\colon Y\to\Delta$ be the projection.
			Properness and flatness are preserved by base change, so $\pi_Y$
			is proper and flat, with central fiber isomorphic
			to $Y_0$.
			The morphism $\mathcal H$ induces a morphism $h\colon M\to Y$
			satisfying
			\[
			\pi_Y\circ h=\pi_M,
			\qquad h|_{M_0}=h_0.
			\]
			Moreover, $h$ is proper since its graph is closed in
			$M\times_\Delta Y$ and the projection
			$M\times_\Delta Y\to Y$ is a base change of the proper morphism
			$\pi_M$.
			
			It remains to verify the contraction assertion.
			Set $\mathcal B:=h_*\mathcal O_M$, and consider the canonical map
			\[
			\alpha\colon\mathcal O_Y\longrightarrow\mathcal B.
			\]
			Flatness of $M$ over $\Delta$ gives an exact sequence
			\[
			0\longrightarrow\mathcal O_M
			\xrightarrow{\ t\ }\mathcal O_M
			\longrightarrow\mathcal O_{M_0}\longrightarrow0.
			\]
			Applying $h_*$ shows that multiplication by $t$ is injective
			on $\mathcal B$ and gives an injection
			\[
			\mathcal B/t\mathcal B
			\hookrightarrow (h_0)_*\mathcal O_{M_0}
			=\mathcal O_{Y_0},
			\]
			where sheaves on $Y_0$ are viewed as sheaves on $Y$.
			The composite
			\[
			\mathcal O_{Y_0}
			=\mathcal O_Y/t\mathcal O_Y
			\xrightarrow{\ \overline{\alpha}\ }
			\mathcal B/t\mathcal B
			\hookrightarrow\mathcal O_{Y_0}
			\]
			is the identity. Hence $\overline{\alpha}$ is an isomorphism.
			
			Since $\mathcal B$ is coherent and $\OO_Y\to \mathcal{B}/t\mathcal{B}$ is surjective, Nakayama's lemma implies that
			$\alpha$ is surjective near $Y_0$.
			For $y\in Y_0$, write
			$K_y=\ker(\alpha_y)$.
			The injectivity of $\overline{\alpha}_y$ gives
			$K_y\subseteq t\mathcal O_{Y,y}$.
			If $k=ta\in K_y$, then $t\alpha_y(a)=0$.
			As multiplication by $t$ is injective on $\mathcal B_y$,
			we have $a\in K_y$.
			Consequently $K_y=tK_y$, and Nakayama's lemma gives $K_y=0$.
			Thus $\alpha$ is an isomorphism on a neighborhood of $Y_0$.
			
			Finally, properness of $\pi_Y$ allows us to shrink $\Delta$
			once more so that this neighborhood contains all of $Y$.
			Therefore $h_*\mathcal O_M=\mathcal O_Y$, as required.
		\end{proof}

		The space $Y$ is normal since $M$ is normal and $h_*\OO_M=\OO_Y$. $\pi_Y:Y\rightarrow \Delta$ is flat, then \cite[p160, Proposition~3.20]{Fischer} implies that the set $\{y\in Y: Y_{\pi_Y(y)} \text{ is not normal at }y\}$ is an analytic set in $Y$. $Y_0$ is normal then implies $Y_t$ are normal for every $t\in\Delta$ after shrinking $\Delta$.

		Consider the exact sequence
		\[0\to\OO_M\xrightarrow{t}\OO_M\to\OO_{M_0}\to0\]
		and we get
		\[\cdots\to R^ih_*\OO_M\xrightarrow{t} R^ih_*\OO_M\to R^ih_*\OO_{M_0}\to\cdots\]
		It is straightforward to verify $(R^ih_*\OO_{M_0})_{y}=(R^i(h_0)_*\OO_{M_0})_y$ for every $y\in Y_0$ and the latter vanishes for $i>0$.
		As a result $(R^ih_*\OO_M)_y\xrightarrow{t} (R^ih_*\OO_M)_y$ is surjective for $i>0$.
		Nakayama's lemma then implies that $(R^ih_*\OO_M)_y=0$. 
		Since $\pi_Y:Y\to \Delta$ is proper, after shrinking $\Delta$ we get
		\begin{equation}\label{eq:descent_O_vanishing}
			R^ih_*\OO_M=0,\qquad i>0.
		\end{equation}

		Now (\ref{eq:descent_O_vanishing}) and Leray spectral sequence give isomorphisms
		\begin{equation}\label{eq:Hi isom}
			h^*:H^i(Y,\OO_Y)\xrightarrow{\sim}H^i(M,\OO_M),\qquad i>0. 
		\end{equation}
		The exact sequence $0\to i\RR\to\OO_Y\xrightarrow{\operatorname{Re}}\PH_Y\to0$ and a similar exact sequence for $M$
		give the natural exact sequence
		\begin{equation}\label{eq:descent_BC_exact}
			\begin{tikzcd}[column sep=large, row sep=large]
				H^1(Y,\mathcal O_Y)
				\arrow[r, "\ell_Y"]
				\arrow[d, "h^*"', "\sim"]
				&
				H^{1,1}_{\mathrm{BC}}(Y,\mathbb R)
				\arrow[r, "j_Y"]
				\arrow[d, "h^*"]
				&
				H^2(Y,\mathbb R)
				\arrow[r]
				\arrow[d, "h^*"]
				&
				H^2(Y,\mathcal O_Y)
				\arrow[d, "h^*", "\sim"']
				\\
				H^1(M,\mathcal O_M)
				\arrow[r, "\ell_M"]
				&
				H^{1,1}_{\mathrm{BC}}(M,\mathbb R)
				\arrow[r, "j_M"]
				&
				H^2(M,\mathbb R)
				\arrow[r]
				&
				H^2(M,\mathcal O_M).
			\end{tikzcd}
		\end{equation}
		We first pass $H^{1,1}_{BC}(\cdot,\RR)$ to $H^2(\cdot,\RR)$ and use the following well-known fact (see e.g. \cite[Chap III Lemma 1.3]{BS})
		\[
		\varinjlim_{\varepsilon\downarrow0}
		H^2\bigl(\pi_Y^{-1}(\Delta_\varepsilon),\RR\bigr)
		\simeq H^2(Y_0,\RR).
		\]
		Thus, after shrinking $\Delta$, $j_{Y_0}(\gamma_0)$ lifts to some $c\in H^2(Y,\RR)$.
		Since $h^*c-j_M(a)$ restricts to zero on $M_0$ we know $h^*c=j_M(a)$ after shrinking $\Delta$ again.
		
		The isomorphism \eqref{eq:Hi isom} for $i=2$ and exact sequence \eqref{eq:descent_BC_exact}
		show that $c$ maps to zero in $H^2(Y,\OO_Y)$, so $c=j_Y(\gamma')$
		for some $\gamma'\in H^{1,1}_{\BC}(Y,\RR)$.
		Since $j_M(a-h^*\gamma^{\prime})=h^*c-h^*(j_Y\gamma^{\prime})=0$, we get $a-h^*\gamma'=\ell_M(v)$ for some $v\in H^1(M,\OO_M)$.
		The isomorphism \eqref{eq:Hi isom} for $i=1$ gives $u\in H^1(Y,\OO_Y)$ with $h^*u=v$ and we set
		$\gamma=\gamma'+\ell_Y(u)$. Thus $h^*\gamma=(a-\ell_Mv)+h^*\ell_Y(u)=a$ as we want.

		Since $h_0:M_0\to Y_0$ is a contraction between compact K\"ahler normal varieties with rational singularities and $R^i(h_0)_*\OO_{M_0}=0$ for $i>0$, Proposition~\ref{prop:mmp-bc-pullback} 
		implies $h_0^*$ is injective on Bott--Chern cohomology classes.
		Therefore $h_0^*(\gamma|_{Y_0}-\gamma_0)=a|_{M_0}-h_0^*\gamma_0=0$ leads to
		$\gamma|_{Y_0}=\gamma_0$. 
		For a smooth representative $\sigma$ of $\gamma$,
		\cite[Proposition 3.3(i)]{DP} gives a smooth function $b$ such that
		$\sigma+\ddc b$ is K\"ahler near $Y_0$.
		The proof is complete after shrinking $\Delta$ again.
	\end{proof}
	
	\begin{theorem}\label{prop:model}
		In the setting of Theorem~\ref{thm:variation}, fix $0\in\Delta$.
		After replacing $\Delta$ by a a smaller disk at $0$ and
		$X$ by its inverse image, retaining the same notation, there exist a
		proper flat family $g:Y\to\Delta$ of normal compact complex spaces,
		a K\"ahler form $\gamma$ on $Y$, and a common resolution in the
		commutative diagram
		\[
		\begin{tikzcd}[column sep=large,row sep=large]
			& Z \arrow[dl,"\mu"'] \arrow[dr,"f"]
			\arrow[dd,"q",description] & \\
			X \arrow[dr,"p"'] & & Y \arrow[dl,"g"] \\
			& \Delta &
		\end{tikzcd}
		\]
		such that $Z$ is K\"ahler and, with $P=f^*\gamma$, there is a real
		effective $f$-exceptional divisor $E_\alpha$ satisfying
		\begin{equation}\label{eq:zariski}
			\mu^*[\alpha]=[P]+[E_\alpha].
		\end{equation}
		Set
		\[
		E_\mu=K_{Z/\Delta}-\mu^*K_{X/\Delta},\qquad
		D=E_\alpha+E_\mu=\sum_{i=1}^{\ell}d_iD_i,
		\]
		where the $D_i$ are the distinct prime components of $D$ and $d_i>0$.
		Then $D$ is effective and $f$-exceptional, and
		\begin{equation}\label{eq:adjointZ}
			[K_{Z/\Delta}+\mu^*\beta]=[P]+[D].
		\end{equation}
		
		For $t\in\Delta$, put
		\[
		Z_t=q^{-1}(t),\quad Y_t=g^{-1}(t),\quad
		\mu_t=\mu|_{Z_t},\quad f_t=f|_{Z_t},\quad
		P_t=P|_{Z_t},\quad D_{i,t}=D_i\cap Z_t.
		\]
		There is a closed discrete subset $T\subset\Delta$ such that, writing
		$\Delta^\circ=\Delta\setminus T$, the map
		$q|_{q^{-1}(\Delta^\circ)}$ is a submersion, the support of
		$D|_{q^{-1}(\Delta^\circ)}$ has relative simple normal crossings,
		and every vertical component of $D$ lies in $q^{-1}(S)$.
		For $t\in\Delta^\circ$, $\mu_t:Z_t\to X_t$ and $f_t:Z_t\to Y_t$ are modifications,
		every irreducible component of $D_{i,t}$ is
		$f_t$-exceptional, $P_t$ is semipositive, $[P_t]$ is big, and
		\begin{equation}\label{eq:orthogonal}
			\int_{D_{i,t}}P_t^{n-1}=0
			\qquad(1\le i\le\ell,\ t\in\Delta^\circ),
		\end{equation}
		where the integral over an empty $D_{i,t}$ is zero.
		Moreover,
		\[
		\vol([\alpha_t])=\vol([\alpha_0])\qquad\text{for every }t\in\Delta.
		\]
	\end{theorem}
	
	\begin{remark}
		$D$ being relative SNC means that for every $x_0\in q^{-1}(\Delta^\circ)$, there is a product coordinate neighborhood $\Omega_z\times V_t$ of $x_0$ so that $\operatorname{Supp}D$ is defined by $z_1\cdot...\cdot z_m$ for some $m\leq n$.
	\end{remark}
	
	\begin{proof}
		We first run MMP for $(X_0,0+\boldsymbol{\beta_0})$, where the b-current $\boldsymbol{\beta_0}$ is induced by $\beta_0=\beta|_{X_0}$. We get a good log terminal model $\phi_0:X_0\dashrightarrow M_0$ so that on a common resolution $X_0\xleftarrow{r_0}W_0\xrightarrow{s_0}M_0$ we have $r_0^*a_{X_0}=s_0^*a_{M_0}+[F_0]$ where $a_{X_0}=[K_{X_0}+\beta_0]:=[\alpha_0]$ denotes the adjoint class of, $a_{M_0}$ denotes that of $M_0$ and $F_0$ is effective $s_0$-exceptional divisor.
		The transendental base-point-free theorem gives a contraction
		$h_0:M_0\to Y_0$ so that $a_{M_0}=h_0^*\gamma$ for some K\"ahler class $\gamma$ on $Y_0$.
		
		We then extend $\phi_0:X_0\dashrightarrow M_0$, $h_0:M_0\to Y_0$ to $\phi:X\dashrightarrow M$, $h:M\to Y$ by Theorem~\ref{lem:mmp-extension} and Lemma~\ref{lem:descent} respectively.
		Take a common resolution $X\xleftarrow{\mu}Z\xrightarrow{r}M$ we obtain $\mu^*[\alpha]=r^*a_M+[E_{\alpha}]$ for some effective $s$-exceptional divisor $E_{\alpha}$.
		
		After removing a discrete set $T\subset\Delta$ we can prove over $\Delta^\circ=\Delta\backslash T$,
		$\mu_t:Z_t\to X_t$ is modification. Moreover the common resolution $X_t\xleftarrow{\mu_t}Z_t\xrightarrow{f_t}Y_t$ gives the divisorial Zariski decomposition $\mu_t^*[\alpha_t]=[P_t]+[E_{\alpha,t}]$ where $P_t=f_t^*\gamma$ is semipositive and big. This implies $\vol([\alpha_t])=\vol([P_t])=\int_{Y_t}\gamma^n$ for $t\in\Delta^\circ$.
		The common resolution $X_0\xleftarrow{s_0}W_0\xrightarrow{h_0\circ r_0}Y_0$ in the first step fulfills that $s_0^*[\alpha_0]=(f_0\circ r_0)^*[\gamma]+[F_0]$ is the divisorial Zariski decomposition.
		Consequently $\vol([\alpha_0])=\int_{Y_0}\gamma^n$.
		Finally we show that $\int_{Y_t}\gamma^n$ is constant on $\Delta$, and this proves $\vol([\alpha_t])$ is also constant.

		\medskip
		\noindent\textbf{Step 1: run MMP for central fiber $X_0$ to get semiample adjoint class.}
		Put $\beta_0=\beta|_{X_0}$ and the adjoint class
		\[
		a_0:=[K_{X_0}+\beta_0]=[\alpha_0],
		\]
		which is big by assumption. 
		Since $\boldsymbol\beta_0$ descends to $X_0$, its generalized
		discrepancies coincide with those of $(X_0,0)$. Thus
		$(X_0,0+\boldsymbol\beta_0)$ has canonical singularity and in particular is gklt.

		By Theorem~\ref{prop:mmp-existence}(=\cite[Theorems 1.1 and 1.5]{HX}), the MMP with scaling gives
		a good log terminal model
		\[
		\phi_0:(X_0,0+\boldsymbol\beta_0)
		\dashrightarrow (M_0,0+\boldsymbol\beta_0),
		\]
		where $M_0$ is normal, compact and K\"ahler.
		The pair $(M_0,0+\boldsymbol\beta_0)$ is gklt by Proposition~\ref{prop:mmp-discrepancies}(1)
		and hence $M_0$ has rational
		singularities by Proposition~\ref{prop:mmp-rational}(1).
		The class
		\[
		a_{M_0}:=[K_{M_0}+\beta_{M_0}]
		\in H^{1,1}_{\mathrm{BC}}(M_0,\mathbb R)
		\]
		is nef by definition of good log terminal model.
		
		Choose a common resolution
		\[
		X_0\xleftarrow{\ r_0\ }W_0\xrightarrow{\ s_0\ }M_0,
		\]
		with $W_0$ smooth, compact and K\"ahler. Then
		$\beta_{M_0}=s_{0*}r_0^*\beta_0$, this proves $\beta_{M_0}$ is modified big. 
		Using Proposition~\ref{prop:mmp-discrepancies},
		good log terminal model $\phi_0:X_0\dashrightarrow M_0$ is $a_0$ non-positive, namely
		\[
		r_0^*a_0=s_0^*a_{M_0}+[F_0],
		\qquad
		F_0\geq0,\quad F_0\text{ is }s_0\text{-exceptional}.
		\]
		The class $r_0^*a_0$ is big by Remark~\ref{rem:mmp-bigness}, 
		Since $s_{0*}[F_0]=0$, we have
		$a_{M_0}=s_{0*}r_0^*a_0$. Consequently, $a_{M_0}$ is big also follows from Remark~\ref{rem:mmp-bigness}.
		
		Applying trancendental base-point-free Theorem~\ref{thm:mmp-basepointfree}
		to
		$(M_0,0+\boldsymbol\beta_0)$, we obtain a Moishezon contraction
		\[
		h_0:M_0\longrightarrow Y_0
		\]
		onto a compact K\"ahler space, together with a K\"ahler class
		$\gamma_0\in H^{1,1}_{\mathrm{BC}}(Y_0,\mathbb R)$ such that
		\[
		a_{M_0}=h_0^*\gamma_0.
		\]
		By \cite[Definition 2.32]{HX},
		$(h_0)_*\mathcal O_{M_0}=\mathcal O_{Y_0}$, it follows that
		$Y_0$ is normal. Since $a_{M_0}$ is big, $h_0$ is bimeromorphic.
		Finally, $a_{M_0}=h_0^*\gamma_0$ also implies $Y_0$ is gklt by Proposition~\ref{prop:mmp-rational}(2) and consequently $Y_0$ has rational singularities.

		\medskip
		\noindent\textbf{Step 2: extend the MMP and $h_0$.}
		Let $\boldsymbol\beta$ be the $b$-current determined by $\beta$ on $X$.
		For $(X_0,0+\boldsymbol\beta_0)$, the hypotheses of
		\cite[Theorem 7.7]{HX} hold: the pair is canonical, the boundary is
		zero, and $N([\alpha_0])\wedge0=0$.
		Therefore the MMP $\phi_0:X_0\dashrightarrow M_0$ extends, after shrinking $\Delta$, to a bimeromorphic contraction between normal spaces
		\[
		\phi:X\dashrightarrow M
		\quad\text{over }\Delta.
		\]
		Let $a_X=[K_{X/\Delta}+\beta_X]$ and $a_M=[K_{M/\Delta}+\beta_M]$ be the adjoint classes, then $a_M|_{M_0}=a_{M_0}$ and the bimeromorphic contraction $X\dasharrow M$ is $a_X$-negative.
		The normal space $M$ is a proper flat family over $\Delta$;
		Lemma~\ref{lem:descent} thus extends $h_0$ to $h:M\to Y$ and gives a
		K\"ahler form $\gamma$ on $Y$ such that
		\[
		a_M=h^*[\gamma],\qquad [\gamma]|_{Y_0}=\gamma_0.
		\]
		
		\medskip
		\noindent\textbf{Step 3: prove the class identities \eqref{eq:zariski} \eqref{eq:adjointZ}.}
		Choose $\mu:Z\to X$ be a sequence of blow-ups to resolve the bimeromorphic map
		$X\dashrightarrow M$, and let $r:Z\to M$ be the induced
		morphism. After shrinking $\Delta$, $Z$ is K\"ahler. Set $f=h\circ r$, $q=p\circ\mu=g\circ f$ and $E_{\mu}=K_{Z/\Delta}-\mu^*K_{X/\Delta}$.

		The bimeromorphic contraction $X\dashrightarrow M$ is $a_X$-negative means 
		\[
		\mu^*[\alpha]=r^*a_M+[E_\alpha]
		=f^*[\gamma]+[E_\alpha],
		\qquad E_\alpha\ge0,\quad E_\alpha\text{ is }r\text{-exceptional}.
		\]
		This proves \eqref{eq:zariski}. 
		Set $D:=E_{\alpha}+E_{\mu}$. $E_{\alpha}$ is $f=h\circ r$-exceptional thanks to it is $r$-exceptional. We may simply assume $D$ is SNC after passing to a further log resolution. Indeed, we take a further log resolution $Z^{\prime}\xrightarrow{\mu'} Z$ so that $\operatorname{Supp}((\mu')^*D)+\operatorname{Exc}(\mu')$ is SNC. If we set $E_{\alpha}'=(\mu')^*E_{\alpha}$, which is SNC, the corresponding class equality \eqref{eq:zariski} still holds on $Z^{\prime}$. 
		Moreover $E_{\mu\circ \mu'}:=K_{Z'/\Delta}-(\mu\circ\mu')^*K_{X/\Delta}=K_{Z'/Z}+(\mu')^*E_{\mu}$ and $D':=E_{\alpha}'+E_{\mu\circ\mu'}=K_{Z'/Z}+(\mu^{\prime})^*(E_{\alpha}+E_{\mu})$ is SNC.
		Since $X\dashrightarrow Y$ is a bimeromorphic contraction, $X\xleftarrow{\mu}Z\xrightarrow{f}Y$ is a common resolution,
		every $\mu$-exceptional divisor is $f$-exceptional. 
		This proves the Jacobian divisor $E_\mu$ is $f$-exceptional. Finally,
		\[
		[K_{Z/\Delta}+\mu^*\beta]=[\mu^*(K_{X/\Delta}+\beta)]+[E_\mu]
		=\mu^*[\alpha]+[E_\mu]=[P]+[D],
		\]
		which proves \eqref{eq:adjointZ}.
		
		\medskip
		\noindent\textbf{Step 4: choose $S$ and prove \eqref{eq:orthogonal}.}
		After shrinking $\Delta$, $D=\sum_{i=1}^ld_iD_i$ has finitely many components. For indices $1\leq i_1<\cdots<i_m\leq\ell$, the critical values of $q|_{D_{i_1}\cap\cdots \cap D_{i_m}}:D_{i_1}\cap\cdots \cap D_{i_m}\to\Delta$ is discrete by Sard theorem. Let $T_1\subset\Delta$ be the set of such values.
		Let $\operatorname{Exc}(f)$ be the exceptional locus of bimeromorphic morphism $f:Z\to Y$, then $\dim \operatorname{Exc}(f)\leq n$ and $\operatorname{Exc}(f)$ has finitely many components after shrinking $\Delta$.
		It follows that there exists a finite set $T_2\in\Delta$ so that $\dim\operatorname{Exc}(f)_t\leq n-1$ for $t\in \Delta\backslash T_2$. Since $\dim Z_t=n$ for every $t\in \Delta$, $f_t:Z_t\to Y_t$ is modification for $t\in \Delta\backslash T_2$.
		Similarly there is a finite set $T_3\subset \Delta$ so that $\dim E_{\mu,t}\leq n-1$ for $t\in \Delta\backslash T_3$ and $\mu_t:Z_t\to X_t$ are modifications for such $t$.
		Let $T_4$ be the set of critical values of $q:Z\to \Delta$.
		We set $T:=T_1\cup T_2\cup T_3\cup T_4$, $\Delta^\circ:=\Delta\backslash T$ then 
		\begin{itemize}
			\item $D$ is relative SNC on $q^{-1}(\Delta^\circ)$,
			\item $f_t:Z_t\to Y_t$ are modifications for $t\in \Delta^\circ$,
			\item $\mu_t:Z_t\to X_t$ are modifications for $t\in \Delta^\circ$,
			\item $q|_{q^{-1}(\Delta^\circ)}$ is a submersion.
		\end{itemize}
		For $t\in\Delta^\circ$, we claim that $D_{i,t}:=D_i\cap q^{-1}(t)$ has image of dimension
		at most $n-2$ under $f_t$. Indeed, from $D_i$ is $f$-exceptional we know
		$\dim f(D_i)\le n-1$. 
		By restricting on $q^{-1}(\Delta^\circ)$, $q(D_i)=\Delta^\circ$ implies $f(D_i)\xrightarrow{g} \Delta^\circ$ is surjective. Its fibers have dimension at most $n-2$ since $f(D_i)$ is irreducible.
		This proves $f(D_{i,t})=(f(D_i))_t$ has dimension at most $n-2$ for $t\in\Delta^\circ $. 
		Note that $D_{i,t}$ are compact manifolds of
		dimension $n-1$ since $t\in \Delta^\circ$. Therefore $\int_{D_{i,t}}P_t^{n-1}=\int_{D_{i,t}}(f_t^*\gamma)^{n-1}=0$ which is 
		\eqref{eq:orthogonal}.
		Since $f_t:Z_t\to Y_t$ is modification, $P_t=f^*_t\gamma$ is semipositive and strictly positive almost everywhere on $Z_t$.
		Since $Z_t$ is K\"ahler we know the class $\{P_t\}$ is big by \cite{DP}.
		
		\medskip
		\noindent\textbf{Step 5: prove the constancy of $\operatorname{vol}([\alpha_t])$ for $t\in\Delta$.}
		
		\emph{Step 5.1: prove $\vol([\alpha_t])=\int_{Z_t}P_t^n=\int_{Y_t}\gamma^n$ for $t\in\Delta^\circ$.} It is clear that
		$\vol([\alpha_t])=\vol(\mu^*[\alpha_t])$ for $t\in \Delta^\circ$, since $\mu_t$ is a modification.
		We claim that $\mu^*[\alpha_t]=[P_t]+[E_{\alpha,t}]=[f^*_t\gamma]+[E_{\alpha,t}]$ is the divisorial Zariski decomposition of $\mu_t^*[\alpha_t]$ with positive part $[P_t]$ and negative part $[E_{\alpha,t}]$.
		Indeed, every irreducible component $E_{i}$ (we omit the subscripts $\alpha,t$) of $E_{\alpha,t}$ is contained in the non-K\"ahler locus $E_{nK}([P_t])$.
		Suppose to the contrary, if there is some $x\in E_i\backslash E_{nK}([P_t])$, then there is a K\"ahler current $P_t+dd^c\varphi$ with analytic singularity which is smooth at $x$. On the other hand $E_i$ is $f_t$-exceptional implies there is a compact subvariety $V_i\subset E_i$ of positive dimension through $x$ s.t. $f_t(V_i)=f_t(x)$.
		However $P_t|_{V_i}=f_t^*\gamma|_{V_i}=0$, we get a psh function $\varphi|_{V_i}$ which is strictly psh at $x$, a contradiction.
		Since $[P_t]$ is big and nef, by \cite[Theorem 3.18]{Boucksom} we know the positive part of $[P_t+E_{\alpha,t}]$ equals $[P_t]$
		and
		\cite[Proposition 3.20]{Boucksom} then gives
		\[
		\vol(\mu^*[\alpha_t])=\vol([P_t+E_{\alpha,t}])=\vol([P_t])=\int_{Z_t}P_t^n
		\]
		and the latter equals $\int_{Y_t}\gamma^n$ since $f_t$ is a modification.
		
		\emph{Step 5.2: prove $\vol([\alpha_0])=\int_{Y_0}\gamma^n$.}
		The idea is same as \emph{Step 5.1}.
		Recall that we take a common resolution $X_0\xleftarrow{r_0}Z_0'\xrightarrow{s_0} M_0$ of the bimeromorphic contraction $X_0\dashrightarrow M_0$ in \noindent\textbf{Step 1}, it holds that
		\[
		r_0^*[\alpha_0]=s_0^*a_{M_0}+[F_0]=(h_0\circ s_0)^*[\gamma_0]+[F_0],
		\qquad
		F_0\geq0,\quad F_0\text{ is }s_0\text{-exceptional},
		\]
		where $\gamma_0$ is the K\"ahler form on $Y_0$ so that $h_0^*[\gamma_0]=a_{M_0}$.
		Clearly $F_0$ is also $h_0\circ s_0$-exceptional and therefore the same argument as \emph{Step 5.1} gives
		\[
		\vol([\alpha_0])=\vol(r_0^*[\alpha_0])=\vol([(h_0\circ s_0)^*\gamma_0])=\int_{Y_0}\gamma_0^n.
		\]
		
		\emph{Step 5.3: prove constancy of $\int_{Y_t}\gamma^n$ on $\Delta$.}
		For proper morphism $g:Y\to \Delta$, the direct image $g_*(\gamma^n\wedge [Y])$ is a closed positive current of degree zero on $\Delta$,
		hence it equals some constant $c$. 
		For any $a\in\Delta$, the defining function of the fiber $Y_a$ is $g-a:=g^*t-a$.
		$Y_a$ is reduced since it is normal by Lemma~\ref{lem:descent}.
		$Y$ is normal so its singular locus has codimension at least two. Since $Y_a$ is of pure codimension one, every component of $Y_a$ intersects $Y_{\text{reg}}$ and reducedness gives us $dd^c\log|g-a|\wedge[Y]=[Y_a]$ (the multiplicity of $[Y_a]$ is one).
		The
		projection formula therefore gives us
		\[
		\begin{aligned}
			g_*\bigl([Y_a]\wedge\gamma^n\bigr)
			&=g_*\ddc\bigl([Y]\wedge\gamma^n\log|g-a|^2 \bigr)\\
			&=g_*([Y]\wedge \gamma^n) \wedge\ddc(\log|t-a|^2) 
			=c\,\delta_a.
		\end{aligned}
		\]
		The left side is
		$\bigl(\int_{Y_a}\gamma_a^n\bigr)\delta_a$, so
		$\int_{Y_a}\gamma_a^n=c$ for every $a\in\Delta$.
		
		We can deduce that $\vol([\alpha_t])=\vol([\alpha_0])$ for $t\in\Delta^\circ$. Since $T$ is discrete, we know $\vol([\alpha_t])$ is constant near $0$ regardless of whether $0\in\Delta^\circ$ or not.
		Since $0$ is chosen arbitrarily, we actually prove that
		$\vol([\alpha_t])$ is locally constant. 
		The proof is therefore complete.
	\end{proof}
	
	\section{Uniform estimates of families of Monge--Amp\`ere equations}\label{sec:uniform estimate of MA}

	We summary the estimates, stability results of Monge--Amp\`ere equations used in this article.
	Let $(X,\omega)$ be a compact K\"ahler manifold of dimension $n$, $\beta$ be a closed semipositive $(1,1)$-form so that $\{\beta\}\in H^{1,1}(X,\RR)$ is big. $\vol([\beta])=\int_X\beta^n>0$.
	Let $dV$ be a positive smooth volume form and $f\geq0$, $f\in L^p(X,dV)$ for some $p>1$ so that $\mu:=fdV$ is a probability measure.
	\begin{theorem}[{\cite[Theorem 2.6]{GL25}}]\label{thm:GL stability}
		Fix $m>n$, let 
		\[A_m(\mu):=\sup\Big\{ \big(\int_X|\psi|^m\mu)^{\frac{1}{m}}   :\psi\in\operatorname{PSH}(X,\beta),\ \sup_X\psi=0\Big\}.\]
		If $\varphi\in\operatorname{PSH}(X,\beta)$ solves
		\[\frac{(\beta+dd^c\varphi)^n}{\vol([\beta])}=\mu, \qquad \sup_X\varphi=0,\]
		then for any $\phi\in\operatorname{PSH}(X,\beta)\cap L^{\infty}(X)$,
		\[\sup_X(\phi-\varphi)_+\leq T\Big(\int_X(\phi-\varphi)_+d\mu \Big)^{r}\]
		where $r$ only depends on $(m,n)$, $T$ only depends on $(m,n,A_m(\mu),\|\phi\|_{L^{\infty}(X)})$.
	\end{theorem}
	
	\begin{theorem}[{\cite[Theorem 1.1]{LWZ}}]\label{thm:LWZ}
		Assume $c,C>0$ so that for every $\psi\in \operatorname{PSH}(X,\beta)$, $\sup_X\psi=0$
		\[\int_X|\psi|dV\leq C,\qquad \int_Xe^{-c\psi}\leq C.\]
		Then the unique solution $\varphi\in\operatorname{PSH}(X,\beta)\cap L^{\infty}(X)$ of the equation
		\[\frac{(\beta+dd^c\varphi)^n}{\vol([\beta])}=fdV,\qquad \sup_X\varphi=0,
		\]
		and $\norm{\varphi}_{L^{\infty}(X)}$ only depends on $(p,n,c,C,\norm{f}_p,\vol([\beta]))$. 
	\end{theorem}
	We point out that in Di Nezza--Guedj--Guenancia \cite[Theorem 1.1]{DNGG} it is proved the estimate of $\norm{\varphi}_{L^\infty(X)}$ is independent of $\vol([\beta])$.
	
	\vspace{0.4cm}
	Let $\theta$ be a closed real $(1,1)$-form on $X$ so that $[\theta]\in H^{1,1}(X,\RR)$ is big and set
	\[
	V_\theta=\sup\{v\in\PSH(X,\theta):v\le0\}
	\]
	and $v_\theta=\vol([\theta])>0$. We need the uniform $L^{\infty}$-estimate for big classes due to Di Nezza--Guedj--Guenancia \cite{DNGG}.
	
	\begin{theorem}[{\cite[Theorem 1.9]{DNGG}}]\label{thm:uniform}
		Assume $dV$ and $fdV$ are probability measures, and following conditions hold 
		\begin{itemize}
			\item There exist $c,C>0$ so that $\int_Xe^{-c\psi}\leq C$ for $\psi\in\operatorname{PSH}(X,\theta)$, $\sup_X\psi=0$,
			\item $(\int_Xf^pdV)^{\frac{1}{p}}\leq C'$ for some $C'>0$.
		\end{itemize}
		Let $\varphi$ be the unique $\theta$-psh function with minimal singularities so that
		\[\frac{\np{(\theta+dd^c\varphi)^n} }{\vol([\theta])}=fdV,\]
		then $V_\theta-M\le \varphi\le V_\theta$. The constant $M$ depends only
		on $p,n,c,C,C'$.
	\end{theorem}
	
	Finally, we present the uniform Skoda's integrability theorem due to Di Nezza--Guedj--Guenancia \cite{DNGG}.
	\begin{theorem}\label{thm:DNGG uniform Skoda}(\cite[Proposition 3.3, Theorem 3.4]{DNGG})
		Let $p:X\to \Delta$ be a submersion to the unit disk where $(X,\omega)$ is K\"ahler. There exist $c,C>0$ so that for every $t\in \Delta_{\frac{1}{2}}$, $\psi_t\in \operatorname{PSH}(X_t,\omega_t)$ with $\sup_{X_t}\psi_t=0$ it holds that
		\[
		\int_{X_t}|\psi_t|\omega^n_t\leq C,\qquad \int_{X_t}e^{-c\psi_t}\omega^n_t\leq C.
		\]
		
	\end{theorem}
	
	We retain the notations in Theorem~\ref{prop:model}.
	Fix a disk $B$ with $\overline B\Subset
	\Delta^\circ$.
	Fix a K\"ahler form $\eta$ on $Z$.
	For each $D_i$, choose a smooth Hermitian metric $h_i$ on
	$\OO_Z(D_i)$ and let $s_i$ be its canonical section. Put
	$\Theta_D=\sum_i d_i\Theta_{h_i}(D_i)$. By
	\eqref{eq:adjointZ}, there is some $h\in C^{\infty}(Z,\mathbb{R})$ such that
	\begin{equation}\label{eq:h}
		\theta_\eta+b=P+\Theta_D+\ddc h,
		\qquad b=\mu^*\beta\ge0.
	\end{equation}
	For $\varepsilon,\tau>0$ we set
	\begin{equation}\label{eq:f}
		f:=e^{-h}\prod_i|s_i|_{h_i}^{2d_i},\qquad
		P_\varepsilon:=P+\varepsilon\eta,\qquad
		f_\tau:=e^{-h}\prod_i(|s_i|_{h_i}^2+\tau)^{d_i}.
	\end{equation}
	Then $f_\tau>0$ is smooth, $f_\tau\downarrow f$ uniformly on
	$q^{-1}(\overline B)$ as $\tau\downarrow0$, and $f>0$ outside $D$.
	Let $u_{\varepsilon,\tau,t}$ be the unique
	smooth solution of
	\begin{equation}\label{eq:smoothMA}
		(P_{\varepsilon,t}+\ddc u_{\varepsilon,\tau,t})^n
		=e^{u_{\varepsilon,\tau,t}}f_{\tau,t}\eta_t^n.
	\end{equation}
	Existence and uniqueness follow from the theorem of Aubin and Yau
	\cite{Aubin,Yau}. These solutions define
	a jointly smooth function $u_{\varepsilon,\tau}(z,t)$ near
	$q^{-1}(\overline B)$.

	\begin{lemma}\label{lem:C0}
		There is a constant $C_B$ independent of $0<\varepsilon,\tau\le1$
		such that
		\[
		\norm{u_{\varepsilon,\tau}}_{L^\infty(q^{-1}(\overline B))}\le C_B.
		\]
		For fixed $\varepsilon$, the functions $u_{\varepsilon,\tau}$ converge uniformly on $q^{-1}(\overline B)$
		to a continuous function $u_\varepsilon$.
		As $\varepsilon\downarrow0$, $u_\varepsilon$ decreases pointwise on
		$q^{-1}(\overline B)$ to a bounded upper-semicontinuous function $u$.
		Moreover, for every $t\in\overline B$, the function $u_t$ is the unique bounded
		$P_t$-psh solution of
		\begin{equation}\label{eq:residualMA}
			(P_t+\ddc u_t)^n=e^{u_t}f_t\eta_t^n.
		\end{equation}
	\end{lemma}
	
	\begin{proof}
		We first establish the uniform estimate for $|u_{\varepsilon,\tau}|$ by uniform integrability Theorem~\ref{thm:DNGG uniform Skoda} of smooth K\"ahler family and $L^{\infty}$ estimate Theorem~\ref{thm:LWZ}.
		Then the stability Theorem~\ref{thm:GL stability} proves the uniform convergence of $u_{\varepsilon,\tau}$ as $\tau\to0$, $\varepsilon$ being fixed.
		Finally comparison principle proves $u_{\varepsilon}$ is decreasing as $\varepsilon\to0$ and the limit $u$ satisfies Monge--Amp\`ere equation \eqref{eq:residualMA} by Bedford--Taylor convergence theorem.
		
		\noindent\textbf{Step 1: the uniform $L^{\infty}$-estimate of solutions.}
		There is a fixed $C_0$ such that $0\le P_\varepsilon\le C_0\eta$
		on $q^{-1}(\overline B)$ for $0<\varepsilon\le1$. Hence every
		$P_{\varepsilon,t}$-psh function is $C_0\eta_t$-psh.
		Apply the uniform integrability Theorem~\ref{thm:DNGG uniform Skoda} we know that, for
		$w\in\PSH(Z_t,P_{\varepsilon,t})$ with $\sup_{Z_t}w=0$, $t\in\overline{B}$
		\begin{equation}\label{eq:uniform Skoda}
			\int_{Z_t}|w|\eta_t^n\le C_1,
			\qquad
			\int_{Z_t}e^{-c_1w}\eta_t^n\le C_1
		\end{equation}
		holds for some uniform constants $c_1,C_1>0$.
		
		Write $u_{\varepsilon,\tau,t}=M_{\varepsilon,\tau,t}+w_{\varepsilon,\tau,t}$ with
		$M_{\varepsilon,\tau,t}:=\sup_{Z_t}u_{\varepsilon,\tau,t}$.
		x
		We first prove $M_{\varepsilon,\tau,t}$ is uniformly bounded. 
		The functions $f_\tau$ are uniformly bounded above, and
		\[
		0<c_2\le I_{\tau,t}:=\int_{Z_t}f_{\tau,t}\eta_t^n\le C_2.
		\]
		Indeed, $f_\tau\ge f$, the restriction of $f$ to each fiber is
		positive outside the divisor $D_t$, and
		$t\mapsto\int_{Z_t}f_t\eta_t^n$ is positive and continuous on
		$\overline B$. Jensen's inequality for the probability measure
		$I_{\tau,t}^{-1}f_{\tau,t}\eta_t^n$ now gives
		\[
		\int_{Z_t}e^{w_{\varepsilon,\tau,t}} f_{\tau,t}\eta_t^n
		\ge I_{\tau,t}
		\exp\left(I_{\tau,t}^{-1}
		\int_{Z_t}w_{\varepsilon,\tau,t} f_{\tau,t}\eta_t^n\right)\ge c_3>0.
		\]
		It is also clear that $\int_{Z_t}e^{w_{\varepsilon,\tau,t}} f_{\tau,t}\eta_t^n\leq C_2$. 
		Moreover,
		\[
		e^{M_{\varepsilon,\tau,t}}
		\int_{Z_t}e^w f_{\tau,t}\eta_t^n
		=V_{\varepsilon,t}:=\int_{Z_t}P_{\varepsilon,t}^n,
		\qquad
		0<v:=\int_{Z_t}P_t^n\le V_{\varepsilon,t}\le C_3,
		\]
		where $v$ is independent of $t$ by Theorem~\ref{prop:model}.
		Thus $|M_{\varepsilon,\tau,t}|$ is uniformly bounded.
		Apply the $L^\infty$-estimate Theorem~\ref{thm:LWZ} to $w_{\varepsilon,\tau,t}$.
		The density functions $e^{u_{\varepsilon,\tau,t}}f_{\tau,t}$  have a uniform $L^p$ bound for any fixed
		$p>1$ and the volumes $\vol([P_{\varepsilon,t}])$ is uniformly bounded.  Consequently $\norm{w_t}_{L^\infty(Z_t)}$
		is uniformly bounded, proving our first assertion.

		\medskip
		\noindent\textbf{Step 2: let $\tau\downarrow0$.}
		Fix $\varepsilon>0$. 
		If $0<\tau'<\tau\leq 1$, then $f_{\tau'}\leq f_{\tau}$ and comparison principle (see e.g. \cite[Lemma 2.5]{DDNL18}) imply $u_{\varepsilon,\tau',t}\geq u_{\varepsilon,\tau,t}$ for every $t$.
		The probability measure $\mu_{\tau,t}:=\frac{e^{u_{\varepsilon,\tau,t}}f_{\tau,t}\eta_t^n}{V_{\varepsilon,t}}\leq \frac{e^{C_B}L\eta_t^n}{v}$ where $L\geq\norm{f_1}_{L^{\infty}}\geq\norm{f_{\tau}}_{L^{\infty}}$.
		Using uniform integrability \eqref{eq:uniform Skoda} we get to know 
		\[A_{n+1}(\mu_{\tau,t})^{n+1}:=\sup_{\psi\in\operatorname{PSH}(Z_t,\eta_t)}\Big\{ \int_{Z_t}|\psi_t-\sup_{Z_t}\psi_t|^{n+1}d\mu_{\tau,t}\Big\} \leq \frac{e^{C_B}L}{v}\frac{(n+1)!}{c_1^{n+1}}C_1.
		\]
		From $u_{\varepsilon,\tau',t}\geq u_{\varepsilon,\tau,t}$, the stability estimate Theorem~\ref{thm:GL stability} gives 
		\[
		\sup(u_{\varepsilon,\tau',t}-u_{\varepsilon,\tau,t})
		=\sup(u_{\varepsilon,\tau',t}-u_{\varepsilon,\tau,t})_+
		\leq T\Big(\int_{Z_t}(u_{\varepsilon,\tau',t}-u_{\varepsilon,\tau,t})_+d\mu_{\tau,t}\Big)^\gamma
		\]
		where $\gamma$ depends only on $n$ and $T$ depends only on $n,v,c_1,C_1,C_B,L$.
		We compute that
		\[
		\int_{Z_t}(u_{\varepsilon,\tau',t}-u_{\varepsilon,\tau,t})d\mu_{\tau,t}
		\leq 
		\int_{Z_t}(e^{u_{\varepsilon,\tau',t}-u_{\varepsilon,\tau,t}}-1)d\mu_{\tau,t}
		=\frac{1}{V_{\varepsilon,t}}\int_{Z_t}(e^{u_{\varepsilon,\tau',t}}-e^{u_{\varepsilon,\tau,t}})f_{\tau,t}\eta_t^n.
		\]
		Using the equations we easily see $\int_{Z_t}e^{u_{\varepsilon,\tau,t}}f_{\tau,t}\eta_t^n=\int_{Z_t}e^{u_{\varepsilon,\tau',t}}f_{\tau',t}\eta_t^n=V_{\varepsilon,t}$, therefore 
		\[
		\sup(u_{\varepsilon,\tau',t}-u_{\varepsilon,\tau,t})
		\leq 
		T\Big(\frac{1}{V_{\varepsilon,\tau}}\int_{Z_t}e^{u_{\varepsilon,\tau',t}}(f_{\tau,t}-f_{\tau',t})\eta_t^n \Big)^{\gamma}\leq T\big(\frac{e^{C_B}}{v}\norm{f_{\tau,t}-f_{\tau',t}}_{L^1(Z_t,\eta_t^n)}\big)^{\gamma}
		\]
		
		\[
		\norm{u_{\varepsilon,\tau,t}-u_{\varepsilon,\tau',t}}
		_{L^\infty(Z_t)}
		\le T
		\norm{f_{\tau,t}-f_{\tau',t}}_{L^p(Z_t,\eta_t^n)}^{\gamma}.
		\]
		Since $f_\tau\to f$ uniformly on $q^{-1}(\overline B)$, this proves the
		uniform convergence of $u_{\varepsilon,\tau}$ on
		$q^{-1}(\overline B)$. The limit $u_\varepsilon$ is therefore continuous,
		and the Bedford--Taylor convergence theorem \cite{BT82} gives
		\[
		(P_{\varepsilon,t}+\ddc u_{\varepsilon,t})^n
		=e^{u_{\varepsilon,t}}f_t\eta_t^n,
		\qquad t\in \overline{B}.
		\]
		
		\medskip
		\noindent\textbf{Step 3: let $\varepsilon\downarrow0$.}
		If $\varepsilon'>\varepsilon$, then $u_{\varepsilon,t}$ is a
		subsolution for the corresponding $\varepsilon'$-equation.
		The comparison principle gives
		$u_{\varepsilon,t}\le u_{\varepsilon',t}$.
		The decreasing limit $u$ is upper semicontinuous.
		The Bedford--Taylor convergence theorem therefore proves \eqref{eq:residualMA}.
	\end{proof}
	
	\vspace{0.4cm}
	For a continuous function $\psi$ on $q^{-1}(\partial B)$, we define the Perron envelope with boundary value $\psi$ as follows.
	\[
	\mathcal P_{\varepsilon,B}[\psi]
	:=\left(\sup\left\{V\in\PSH(q^{-1}(B),P_\varepsilon):
	\limsup_{q^{-1}(B)\ni x\to\xi}V(x)\le\psi(\xi)
	\text{ for every }\xi\in q^{-1}(\partial B)\right\}\right)^*.
	\]
	For fixed $\varepsilon,\tau>0$, let $U_{\varepsilon,\tau}$ be the Perron envelope with boundary value $u_{\varepsilon,\tau}|_{q^{-1}(\partial B)}]$. 
	By Chu--Tosatti--Weinkove \cite[Corollary 1.3]{CTW17} (see also \cite[Theorems 1.1 and 2.1]{CTW18}), $U_{\varepsilon,\tau}\in C^{1,1}(q^{-1}(\overline B))$ and satisfies
	\begin{equation}\label{eq:HCMA}
		T:=P_\varepsilon+\ddc U_{\varepsilon,\tau}\ge0,
		\qquad T^{n+1}=0\quad\text{on }q^{-1}(B),
		\qquad
		U_{\varepsilon,\tau}|_{q^{-1}(\partial B)}
		=u_{\varepsilon,\tau}|_{q^{-1}(\partial B)}.
	\end{equation}
	
	\begin{lemma}\label{lem:UC0}
		For every $\varepsilon,\tau>0$ we have
		\begin{equation}\label{eq:UC0}
			|U_{\varepsilon,\tau}|\leq C_B+AR^2:=C_{B}^{\prime}
		\end{equation} 
		where $C_B$ is in Lemma~\ref{lem:C0} and $A$ only depends on $P$ and $\eta$.
	\end{lemma}
	\begin{proof}
		The constant $-C_B$ in Lemma~\ref{lem:C0} is clearly a candidate of $U_{\varepsilon,\tau}$ thus $U_{\varepsilon,\tau}\geq-C_B$.
		
		On the other hand,
		set
		\[
		\overline U:=C_B+A\bigl(R^2-|q-b_0|^2\bigr).
		\]
		It is clear that $U_{\varepsilon,\tau}\leq C_B=\overline{U}$ on $q^{-1}(\partial B)$.
		A sufficiently large $A$ makes $P_\varepsilon+\ddc\overline U$ negative in the horizontal direction at every point of $q^{-1}(\overline B)$.
		If $U_{\varepsilon,\tau}>\overline U$ somewhere, then $U_{\varepsilon,\tau}-\overline U$ attains a positive maximum at an
		interior point. 
		Let $A_{\varepsilon}$ be a local potential of $P_\varepsilon$,
		the smooth function $A_{\varepsilon}+\overline U$ would touch
		psh function $A_{\varepsilon}+U_{\varepsilon,\tau}$ from above at that point. The complex Hessian $P_{\varepsilon}+dd^c\overline{U}$ of $A_{\varepsilon}+\overline{U}$ must be
		nonnegative, a contradiction. 
		Hence $U_{\varepsilon,\tau}\le\overline U$. This proves the uniform estimate of $U_{\varepsilon,\tau}$.
	\end{proof}
	Let $\hat{u}:=\lim_{\tau\to0}\lim_{\varepsilon\to0}U_{\varepsilon,\tau}$. We will finish the proof of plurisubharmonicity of $u$ by showing that for every $a\in B$
	\begin{equation}\label{equ:sup solution}
		(P+dd^c\hat{u}_a)^n\leq e^{\hat{u}_a}f_a\eta_a^n.
	\end{equation}

	\section{An estimate of Perron envelope by Bergman kernels}\label{sec:Perron estimate}
	Throughout this section 
	we fix $\varepsilon,\tau>0$, set
	\[
	U:=U_{\varepsilon,\tau},\qquad
	T:=P_\varepsilon+\ddc U,\qquad S:=T^n.
	\]
	To prove \eqref{equ:sup solution}, we need to establish a upper bound of $\log^+\frac{T_a^n}{e^{U_a}f_a\eta^n_a}$ (see Lemma~\ref{lem:green}).
	We need the following simple observation.
	
	\begin{lemma}\label{lem:kernel}
		Let $\Omega^+\times V\subset q^{-1}(B)$ be a product coordinate
		chart of $Z$, with coordinates $(z,t)$ such that $q(z,t)=t$, and choose a ball $\Omega\Subset\Omega^+$. 
		On this chart we choose $a_\varepsilon$ with $\ddc a_\varepsilon=P_\varepsilon$ and put
		$v:=a_\varepsilon+U$. Write $\eta_t^n=\rho(z,t)d\lambda_z$.
		For $t\in V$, let $K_{kv}(z,t)$ be the weighted Bergman kernel
		on $\Omega$, namely
		\[
		K_{kv}(z,t)=
		\sup\left\{|F(z)|^2:F\in\OO(\Omega),\quad
		\int_\Omega |F|^2e^{-kv(\cdot,t)}d\lambda_z\le1\right\}.
		\]
		Using $c_n=n!$, we define a function $F_k$ on
		$\Omega\times V$
		\[
		F_k:=\log\frac{c_nk^{-n}K_{kv}e^{-kv}}{e^Uf\rho},
		\qquad F_k^+:=\max\{F_k,0\}.
		\]
		Then $F_k^+$ is locally integrable with respect to $S$ and
		\begin{equation}\label{eq:currentineq}
			\ddc(F_k^+S)\ge- S \wedge [D].
		\end{equation}
	\end{lemma}
	\begin{remark}
		The Monge--Amp\`ere product $S\wedge[D]:=(dd^cv)^n\wedge [D]$ is well-defined because $v$ is locally bounded.
		The coordinate chart $\Omega^+\times V$ is allowed to be adjacent to $\partial B$, which means $\partial V$ may attach $\partial B$.
	\end{remark}
	\begin{proof}
		Let $f_i$ be the defining function of $D_i$ on $W$, and set
		\[
		\ell_D:=\sum_{i}d_i\log|f_i|^2,
		\qquad
		\Psi_\varepsilon:=a_\varepsilon+\ell_D-\log(f\rho).
		\]
		A direct computation by \eqref{eq:h} shows that $\ddc\Psi_\varepsilon=b+\varepsilon\eta\ge0$.
		Berndtsson's theorem \cite[Theorem 1.1]{Ber06} shows $\log K_{kv}$ is plurisubharmonic in $(z,t)$. Consequently
		\[
		W_k=\max\{\log(c_nk^{-n}K_{kv})+\Psi_\varepsilon,
		(k+1)v+\ell_D\}
		\]
		is plurisubharmonic and locally bounded.
		Since $v=a_{\varepsilon}+U$ is $C^{1,1}$, the coefficients of $S=(dd^cv)^n$ are locally bounded, so $\ell_D$ is integrable against $S$. 
		We compute that
		\[
		\ddc(F_k^+S)
		=\ddc W_k\wedge S-(k+1)T^{n+1}-S \wedge [D]
		\ge-S \wedge [D].
		\]
		Here we used both $\ddc W_k\wedge S\ge0$, the homogeneous
		Monge--Amp\`ere equation $T^{n+1}=0$ and
		the identity $\ddc(\ell_DS)=S \wedge [D]$, which can be obtained by regularizing $v$.
	\end{proof}
	\begin{remark}
		In Lemma~\ref{lem:kernel}, the Bergman kernel function $\log K_{kv}$ can be replaced by any bounded plurisubharmonic function. The property of $\log K_{kv}$ is used later.
	\end{remark}
	
	\vspace{0.3cm}
	Fix $a\in B=B(b_0,R)$, 
	the (non-negative) Green function of $B$ with pole at $a$ is defined as
	\[
	g_a(t):=
	\begin{cases}
		\displaystyle\log\left|
		\frac{R-\frac{\overline{(a-b_0)}(t-b_0)}{R}}{t-a}\right|^2,
		&t\in B,\\[4pt]
		0,&t\in\CC\setminus B.
	\end{cases}
	\]
	$g_a$ satisfies on $\CC$,
	\[
	\ddc g_a=\nu_a-\delta_a,\qquad
	d\nu_a(b_0+Re^{i\theta})=
	\frac{1-|a_B|^2}{|e^{i\theta}-a_B|^2}
	\frac{d\theta}{2\pi}.
	\]
	$\nu_a$ is a probability measure supported on $\partial B$.
	In a neighborhood of $\overline B$, $g_a=O(1+|\log|t-a||)$ and
	$|dg_a|=O(1+|t-a|^{-1})$.
	
	\vspace{0.4cm}
	Write
	\[
	\mu_a=T_a^n,\qquad \lambda_a=e^{U_a}f_a\eta_a^n.
	\]
	They are measures on $Z_a$ with bounded density functions. 
	The following estimate is one of the technical hearts of our proof.
	\begin{lemma}\label{lem:green}
		For every $a\in B$,
		\begin{align}\label{eq:entropy}
			\int_{Z_a}\log^+\!\left(\frac{d\mu_a}{d\lambda_a}\right)d\mu_a
			\le{}&
			\int_{\partial B}\left[
			\int_{Z_t}e^{u_{\varepsilon,\tau,t}}f_{\tau,t}
			\log\!\left(\frac{f_{\tau,t}}{f_t}\right)\eta_t^n
			\right]d\nu_a(t)\notag\\
			&+\int_{q^{-1}(B)}(q^*g_a)\,T^n\wedge[D].
		\end{align}
	\end{lemma}
	
	\begin{proof}
		
		The key ingredient is Berman--Berndtsson's \cite[Theorem 2.1 and Corollary 2.2]{BB}:
		\begin{theorem}\label{thm:BB}
			Let $B$ be the unit ball in $\CC^n$, $v$ be a bounded plurisubharmonic function defined on $B_{1+\delta}$ so that $\Delta v\in L^{\infty}_{\mathrm{loc}}(B_{1+\delta})$.
			For $k=1,2,\cdots$, let $K_{kv}$ be the weighted Bergman function on $B_{1+\delta/2}$ and 
			\[
			\beta_k:=c_nk^{-n}K_{kv}e^{-kv}d\lambda_z
			\]
			be the Bergman measures.
			Then there is a constant $C$ only depending on $\norm{\Delta v}_{L^{\infty}(B_{1+\delta/2})}$ so that $c_nk^{-n}K_{kv}e^{-kv}\leq C$ for every $k$, moreover $\norm{\beta_k-(dd^cv)^n}_{L^1(B)}\to0$ as $k\to\infty$.
		\end{theorem}
		
		\noindent\textbf{Step 1: Convergence and uniform bound of weighted Bergman kernels on each fiber $Z_t$.}
		Choose finitely many product coordinates $\Omega^+_j\times \overline{V_j}$ to cover
		$q^{-1}(\overline B)$ so that they are adapted to the relative SNC divisor $D$: the defining function of $D$ does not involve the horizontal coordinate function $t_j$.
		We may assume for some smaller balls $\Omega_j\subset\subset\Omega_j^{\prime}\subset\subset\Omega_j^+$, $\Omega_j\times\overline{V}_j$ still cover $q^{-1}(\overline{B})$.
		Let $K_{kv_j}$ be the weighted Bergman kernel on $\Omega_j^{\prime}\times \overline{V_j}$.
		Choose a partition of unity $\{\chi_j\}$ so that $\text{supp}(\chi_j)\subset \Omega_j\times \overline{V_j}$.

		On each $\Omega_j\times \overline{V_j}$ and $t\in\overline{V_j}$, write
		\[
		b_{j,k}(\cdot,t):=c_nk^{-n}K_{kv_j}(\cdot,t)e^{-kv_j(\cdot,t)},\qquad
		T_t^n:=r_j(\cdot,t)d\lambda_z,\qquad
		e^{U_t}f_t\eta_t^n:=d_j(\cdot,t)d\lambda_z.
		\]
		Since $v_j\in C^{1,1}(\Omega_j^+\times \overline{V_j})$, by Berman--Berndtsson' theorem we get to know
		\begin{equation}\label{eq: Bergman bound}
			\norm{b_{j,k}(\cdot,t)-r_j(\cdot,t)}_{L^1(\Omega_j,d\lambda_z)}
			\longrightarrow0,\qquad 0\le b_{j,k},r_j\le C_{\varepsilon,\tau}
		\end{equation}
		where the constant $C_{\varepsilon,\tau}$ only depends on $\norm{dd^c_zv_j}_{L^{\infty}(\Omega_j^{\prime}\times\overline{V_j})}\leq \norm{U}_{C^{1,1}(q^{-1}(\overline{B}))}$.
		
		By definition of $F_{j,k}$ in Lemma~\ref{lem:kernel} we have
		\[
		F_{j,k}^+=\log^+(b_{j,k}/d_j)
		\]
		and set $F^+:=\log^+(r_j/d_j)=\log^+(\frac{T^n_t}{e^{U_t}f_t\eta_t^n})$, which is actually independent of the coordinate chart. 
		Since $d_j>0$
		off $D$, the preceding $L^1$ convergence \eqref{eq: Bergman bound} implies the $L^1$ convergence
		of $F_{j,k}^+\to F^+$ on each fiber and the uniform upper bound:
		\[
		0\le F_{j,k}^+,F^+\le
		M_j:=C_{\varepsilon,\tau}
		+\sum_{i}d_i\log^-|f_{j,i}|^2 \qquad \text{on $\Omega_j\times \overline{V_j}$}.
		\]
		where $f_{j,i}$ are defining functions of divisors $D_i$ on $\Omega_j\times V_j$.
		Since $D$ is relative SNC, each $f_{j,i}$ only involves vertical coordinate functions $z_j$. Consequently
		$M_j=M_j(z_j)$ is uniformly integrable in $z$-variable for $t\in\overline{V_j}$.
		Dominanted convergence theorem gives
		\[
		\int_{\Omega_{j,t}}|F_{j,k}^+-F^+|\,d\lambda_z\longrightarrow0,
		\qquad
		\int_{\Omega_{j,t}}|F_{j,k}^+-F^+|\,T_t^n\longrightarrow0,
		\]
		the second assertion uses $T_t^n\le C_{\varepsilon,\tau}d\lambda_z$.
		
		\medskip
		\noindent\textbf{Step 2: Convergence of Bergman kernel on the total space and on the boundary.}
		The fiber integrals in \emph{Step 1.} are bounded uniformly
		in $k,t$. We prove that in each
		product chart $\Omega_j\times V_j$,
		\begin{align*}
			\int_{V_j}\left[
			\int_{\Omega_{j,t}}|F_{j,k}^+-F^+|\,d\lambda_z\right]d\nu_a(t)
			&\longrightarrow0, \qquad k\to\infty\\
			\int_{V_j}(1+g_a+|dg_a|)\left[
			\int_{\Omega_{j,t}}|F_{j,k}^+-F^+|\,d\lambda_z\right]dA(t)
			&\longrightarrow0, \qquad k\to \infty.
		\end{align*}
		Here $dA$ is Euclidean measure on $\Delta$.
		In fact
		$1+g_a+|dg_a|\in L^1(B,dA)$, $F_{j,k}^+,F\leq M_j$ on $\Omega_j\times V_j$ and $(1+g_a+|dg_a|)M_j$ is integrable. The $L^1$-convergence of each fiber together with dominated convergence theorem proves two convergences.
		
		\medskip
		\noindent\textbf{Step 3: deduce the upper bound of $\int_{Z_a}\log^+(\frac{T_a^n}{\lambda_a})d\mu_a$.}
		Set 
		\[
		H_k(t)=\sum_j\int_{Z_t}\chi_jF_{j,k}^+T_t^n,\qquad t\in\overline{B}.
		\]
		This function is continuous on $\overline B$. Indeed the potential $U$ is continuous on $q^{-1}(\overline{B})$, so $T_t^n$ converge weakly, moreover the $C^{1,1}$ property of $U$ gives $T_t^n\le C_{\varepsilon,\tau}d\lambda_z$. 
		The Bergman kernels $K_{kv_j}(\cdot,t)$ depend continuously on $t$: if
		$\norm{v_j(\cdot,t)-v_j(\cdot,s)}_\infty\le r$, its definition gives
		$e^{-kr}K_{kv_j}(z,s)\le K_{kv_j}(z,t)\le e^{kr}K_{kv_j}(z,s)$.
		Outside $D$, $F_{j,k}^+$ are
		continuous and we have $F_{j,k}^+\le M_j$ and $M_j$ is uniformly integrable with respect to $t$ by \emph{Step 1.}. This proves the continuity of $H_k$ as we want.
		
		For any nonnegative $h\in C_c^\infty(B)$, an integration by parts yields that
		\begin{align*}
			\int_B H_kdd^ch
			&=\int_Bq_*\bigl(\sum_j\chi_jF_{j,k}^+T^n\bigl) dd^ch\\
			&=\int_{q^{-1}(B)}\Big(\sum_j dd^c(\chi_jF_{j,k}^+T^n)\Big)q^*h.
		\end{align*}
		By \eqref{eq:currentineq}, $dd^c(F_{j,k}^+T^n)\geq-S\wedge[D]$ on every coordinate chart, an integration by parts gives us
		\begin{align*}
			\int_B H_k\,\ddc h+E_k(h)
			&\ge-\int_{q^{-1}(B)}q^*h\,S\wedge[D],\\
			E_k(h)&:=\sum_j\int_{\Omega_j\times V_j}
			F_{j,k}^+S\wedge
			\bigl(q^*h\,\ddc\chi_j+q^*(dh)\wedge d^c\chi_j
			+d\chi_j\wedge q^*(d^ch)\bigr).
		\end{align*}
		This inequality still holds for $h=g_a$,
		the reason is that Green function $g_a$ satisfies $g_a,|dg_a|$ are integrable and $g_a $ vanishes on the boundary $\partial B$.
		We can approximate $g_a$ by a sequence of smooth functions $h_j\geq0$ with compact supports in $B$ so that $\norm{h_j-g_a}_{W^{1,1}(B)}\to0$ and $dd^ch_j\to dd^cg_a$ in the sense of measure.
		Then $\int_BH_kdd^ch_j=\int_{\CC}H_kdd^ch_j$ converge to $\int_{\CC}H_k dd^cg_a=\int_{\partial B}H_k d\nu_a-H_k(a)$.
		The convergence of $E_k(h_j)$ follows from the uniform integrability of $F_{j,k}^+(g_a+|dg_a|)$.
		The convergence $\int_{q^{-1}(B)}(q^*h_j)S\wedge[D]\to \int_{q^{-1}(B)}(q^*g_a)S\wedge[D]$ uses the fact that $S=T^n$ has bounded coefficients on $q^{-1}(B)$.
		As a consequence we obtain:
		\begin{equation}\label{eq:finitegreen}
			H_k(a)\le\int_{\partial B}H_k\,d\nu_a
			+\int_{q^{-1}(B)}q^*g_a\,[D]\wedge S+E_k,
		\end{equation}
		where
		\[
		E_k:=\sum_j\int_{\Omega_j\times V_j}
		F_{j,k}^+S\wedge
		\bigl(q^*g_a\,\ddc\chi_j+q^*(dg_a)\wedge d^c\chi_j
		+d\chi_j\wedge q^*(d^cg_a)\bigr).
		\]
		
		\medskip
		\noindent\textbf{Step 4: pass to the limit.}
		Steps 1 and 2 allow passage to the limit in every term of
		\eqref{eq:finitegreen}. 
		We have
		\begin{align*}      &\lim_{k\to\infty}H_k(a)=\lim_{k\to\infty}\int_{Z_a}\sum_j\chi_jF_{j,k}^+T_a^n=\int_{Z_a}F^+T_a^n,\\
			&\lim_{k\to\infty}\int_{\partial B}H_kd\nu_a=\lim_{k\to \infty}\int_{\partial B}\int_{Z_t}\sum_j\chi_jF_{j,k}^+T_t^nd\nu_a(t)=\int_{\partial B}\int_{Z_t}F^+T^n_td\nu_a(t)\\
			&\lim_{k\to \infty}E_k=\sum_j\int_{ \Omega_j\times V_j}F^+S\wedge\bigl(q^*g_add^c\chi_j+q^*(dg_a)\wedge d^c\chi_j+d\chi_j\wedge q^*(d^cg_a)\bigl)=0.
		\end{align*}
		The last one becomes zero is simply because
		$\sum_j\ddc\chi_j=\sum_jd\chi_j=\sum_jd^c\chi_j=0$. By definition of $F^+,T^n_t$ we have $F^+T^n_t=\log^+(\frac{\mu_t}{\lambda_t})d\mu_t$ on $Z_t$ for every $t\in\overline{B}$.
		For $t\in\partial B$, $U_t=u_{\varepsilon,\tau,t}$ and
		\eqref{eq:smoothMA} give
		\[
		T_t^n=e^{u_{\varepsilon,\tau,t}}f_{\tau,t}\eta_t^n,
		\qquad F^+=\log(f_{\tau,t}/f_t).
		\]
	\end{proof}
	
	We show that the second error term in \eqref{eq:entropy} actually converges to $0$ as $\varepsilon\to0$.
	\begin{lemma}\label{lem:error}
		For each $a\in B$ there is a constant $C_{a,B}$ independent of
		$0<\varepsilon,\tau\le1$ such that
		\begin{equation}\label{eq:error}
			\int_{q^{-1}(B)}(q^*g_a)\,  T^n\wedge[D]
			\le C_{a,B}\sum_i d_i m_{i,\varepsilon},\qquad
			m_{i,\varepsilon}:=\int_{D_{i,t}}P_{\varepsilon,t}^{n-1}.
		\end{equation}
		Each $m_{i,\varepsilon}$ is independent of $t\in B$ and tends to
		zero as $\varepsilon\downarrow0$.
	\end{lemma}
	
	\begin{proof}
		The idea is to use the decomposition $T=P_{\varepsilon}+dd^cU$. We first obtain the upper bound of $\int_{q^{-1}(B)}(q^*g_a)P_{\varepsilon}^n\wedge[D]$. Then $U\in C^{1,1}(q^{-1}(B))$ is sufficient to control the difference $T^n-P_{\varepsilon}^n=dd^c\bigl(U\sum_{j=0}^{n-1} T^j\wedge P_{\varepsilon}^{n-1-j}\bigl)$. Put
		$q_i:=q|_{D_i\cap q^{-1}(B)}$.

		\medskip
		\noindent\textbf{Step 1: constancy of the fiber volumes.}
		The map $q_i$ is proper and a submersion, and
		$(q_i)_*(P_\varepsilon^{n-1})$ is a degree-zero current.
		$(q_i)_*$ commutes with $d$, so this current is closed
		and equals a constant $c_{i,\varepsilon}$. For any $t_0\in B$, $[D_{i,t_0}]=dd^c(q_i^*\log|t-t_0|)$, thus the projection formula give
		\begin{align*}
			(q_i)_*\bigl( P_\varepsilon^{n-1}\wedge [D_{i,t_0}] \bigr)
			&=(q_i)_*(P_{\varepsilon}^{n-1}\wedge dd^cq_i^*\log|t-t_0|)
			= (q_i)_*(P_\varepsilon^{n-1})\ddc\log|t-t_0|\\
			&=c_{i,\varepsilon}\delta_{t_0}.
		\end{align*}
		The left side is
		$\bigl(\int_{D_{i,t}}P_{\varepsilon,t}^{n-1}\bigr)\delta_t$.
		Thus $m_{i,\varepsilon}=c_{i,\varepsilon}$ on every fiber.
		Finally, $P_\varepsilon=P+\varepsilon\eta\to P$ and
		\eqref{eq:orthogonal} give
		\[
		\lim_{\varepsilon\to0}m_{i,\varepsilon}
		=\int_{D_{i,t}}P_t^{n-1}=0.
		\]
		
		\medskip
		\noindent\textbf{Step 2: estimate $\int_{q^{-1}(B)}g_a  P_{\varepsilon}^n\wedge[D]$.}
		Since $D$ is relative SNC, $q_i=q|_{D_i}$ is a submersion.
		In an adapted coordinate we see the coefficient of $i dt\wedge d\bar{t}$ of $P_\varepsilon|_{D_i}=(P+\varepsilon \eta)|_{D_i}$ is uniformly bounded for $0<\varepsilon\le1$.
		It follows that for some $C>0$ depending on $B,P,\eta$ so that
		\[
		P_\varepsilon^n|_{D_i}
		\le Cq_i^*\omega_B\wedge P_\varepsilon^{n-1}|_{D_i},
		\]
		where $\omega_B=idt\wedge d\bar{t}$ is the standard K\"ahler form on $B$. By \emph{Step 1},
		\[
		\int_{D_i\cap q^{-1}(B)}q_i^*g_a\,P_\varepsilon^n
		\leq C\int_{D_i\cap q^{-1}(B)}q_i^*(g_a\omega_B)\wedge P_{\varepsilon}^{n-1}
		=C\int_B\big({q_i}_*P_{\varepsilon}^{n-1}\big)g_a\omega_B
		= C m_{i,\varepsilon}\int_Bg_a\omega_B.
		\]

		\medskip
		\noindent\textbf{Step 3: estimate the difference of $T^n,P_{\varepsilon}^n$.}
		Let
		\[
		Q:=\sum_{j=0}^{n-1}T^j\wedge P_\varepsilon^{n-1-j},\qquad
		A_i(t):=\int_{D_{i,t}}U_t
		\sum_{j=0}^{n-1}T_t^j\wedge P_{\varepsilon,t}^{n-1-j}.
		\]
		Then $dd^c(UQ)=T^n-P_{\varepsilon}^n$ and $A_i=(q_i)_*(UQ)$.
		Integration by parts on each compact fiber gives
		\[
		\int_{D_{i,t}}T_t^j\wedge P_{\varepsilon,t}^{n-1-j}
		=m_{i,\varepsilon}\qquad(0\le j\le n-1).
		\]
		The estimate \eqref{eq:UC0} in Lemma~\ref{lem:UC0} consequently implies
		$|A_i(t)|\le nC_B^{\prime}m_{i,\varepsilon}$, $C_B^{\prime}$ depends only on $B,P,\eta$.
		$U$ is continuous on $q^{-1}(\overline{B})$ and Bedford--Taylor
		continuity show that $A_i$ is continuous on $\overline B$.
		Since $dd^c$ commutate with $(q_i)_*$, a direct computation gives
		\[
		\ddc A_i=(q_i)_*\bigl(dd^c(UQ)\bigl)=(q_i)_*(T^n-P_\varepsilon^n).
		\]
		
		Since $T=P_{\varepsilon}+dd^cU$ has bounded coefficients on $q^{-1}(B)$, we know $dd^cA_i$ has bounded coefficient on $B$.
		It follows that
		\begin{align*}
			\int_{D_i\cap q^{-1}(B)}(q_i^*g_a)\,(T^n-P_\varepsilon^n)
			=\int_B g_a (q_i)_*(T^n-P_{\varepsilon}^n)
			&=\int_Bg_a\,\ddc A_i\\
			&=\int_{\partial B}A_i\,d\nu_a-A_i(a).
		\end{align*}
		Hence
		\[
		\left|\int_{D_i\cap q^{-1}(B)}
		q_i^*g_a\,(T^n-P_\varepsilon^n)\right|
		\le 2nC_B^{\prime}m_{i,\varepsilon}.
		\]
		Combining this with \emph{Step 2}, the proof is complete.
	\end{proof}
	
	\section{Proof of plurisubharmonic variation for big adjoint classes}\label{sec:pf of psh variation}
	
	\begin{proposition}\label{prop:analytic}
		We keep the notations in Theorem~\ref{prop:model} and in Section~\ref{sec:uniform estimate of MA}. Then the solutions $u_t$ of equation \eqref{eq:residualMA}
		\[(P_t+dd^cu_t)^n=e^{u_t}f_t\eta_t^n\]
		patch together to give a $P$-psh function $u(z,t):=u_t(z)$ on $q^{-1}(\Delta^\circ)$.
	\end{proposition}
	
	\begin{proof}
		As we have mentioned, we are going to prove that $\hat{u}=\lim_{\varepsilon\to0}\lim_{\tau\to0}U_{\varepsilon,\tau}$ satisfies $(T_a+dd^c\hat{u}_a)^n\leq e^{\hat{u}_a}f_a\eta_a^n$ for eny $a\in B$.
		
		\medskip
		\noindent\textbf{Step 1: estimate $\int_{Z_a}\log^+\frac{(P_{\varepsilon,a}+dd^cU_{\varepsilon,\tau,a})^n}{e^{U_{\varepsilon,\tau,a}}f_{\tau,a}\eta^n_a}$.}
		We first fix a point $a\in B$ and $\varepsilon,\tau>0$, simply put $U=U_{\varepsilon,\tau}$ and
		$T=P_\varepsilon+\ddc U$.
		We denote the first error term in \eqref{eq:entropy} by
		\[
		b_{\varepsilon,\tau,a}
		:=\int_{\partial B}\left(\int_{Z_t}
		e^{u_{\varepsilon,\tau,t}}f_{\tau,t}
		\log\frac{f_{\tau,t}}{f_t}\,\eta_t^n\right)d\nu_a(t).
		\]
		Since $f_\tau\ge f$ and $f_\tau\to f$ uniformly,
		$f_\tau\log(f_\tau/f)\geq0$ tends to zero almost everywhere and is bounded above
		by $C(1+|\log f|)$. Relative SNC condition of $D$ makes $C(1+|\log f|)$ integrable over
		$q^{-1}(\partial B)$ with respect to the fiber measures $\eta_t^n$ and boundary measure
		$\nu_a$. The bound $|u_{\varepsilon,\tau}|\leq C_B$ in
		Lemma~\ref{lem:C0} consequently gives
		\[
		\sup_{0<\varepsilon\le1}b_{\varepsilon,\tau,a}\longrightarrow0,
		\qquad \tau\downarrow0.
		\]
		On $Z_a$ write
		$T_a^n:=r\eta_a^n$ and $e^{U_a}f_a\eta_a^n:=s\eta_a^n$.
		For any $\chi\in C(Z_a)$ with $0\le\chi\le1$,
		\[
		\begin{aligned}
			\int_{Z_a}\chi\bigl(T_a^n-e^{U_a}f_a\eta_a^n\bigr)
			&\le\int_{Z_a}(r-s)_+\eta_a^n\\
			&\le\int_{Z_a}r\log^+(r/s)\,\eta_a^n.
		\end{aligned}
		\]
		Lemmas~\ref{lem:green} and~\ref{lem:error} yield
		\begin{equation}\label{eq:weakdefect}
			\int_{Z_a}\chi\,T_a^n
			\le\int_{Z_a}\chi e^{U_a}f_a\eta_a^n
			+b_{\varepsilon,\tau,a}
			+C_{a,B}\sum_i d_i m_{i,\varepsilon}.
		\end{equation}
		
		\medskip
		\noindent\textbf{Step 2: pass to the two limits.}
		For $\varepsilon>0$, define 
		$U_\varepsilon$
		be the Perron envelope with boundary value $u_\varepsilon|_{q^{-1}(\partial B)}$ as in \S~\ref{sec:uniform estimate of MA}.
		It follows that
		\[
		\norm{U_{\varepsilon,\tau}-U_\varepsilon}_{L^\infty(q^{-1}(B))}
		\le
		\norm{u_{\varepsilon,\tau}-u_\varepsilon}_{L^\infty(q^{-1}(\partial B))}
		\longrightarrow0.
		\]
		In particular $U_\varepsilon$ is continuous on
		$q^{-1}(\overline B)$ and has the indicated boundary values.
		
		As $\varepsilon\downarrow0$, both the admissible classes $\text{PSH}(q^{-1}(B),P_{\epsilon})$ and the
		boundary values decrease, so $U_\varepsilon$ decreases pointwise to a function
		$\hat{u}$ on $q^{-1}(\overline B)$. In particular,
		$\hat{u}=u$ on $q^{-1}(\partial B)$, since $U_\varepsilon=u_\varepsilon$
		there for every $\varepsilon$. The uniform bound of $U_{\varepsilon,\tau}$ \eqref{eq:UC0} shows that
		$U$ is bounded and $P$-psh on $q^{-1}(B)$. 
		
		On the fixed fiber $Z_a$, pass first $\tau\downarrow0$ and then
		$\varepsilon\downarrow0$ in \eqref{eq:weakdefect}. Bedford--Taylor convergence theorem \cite{BT82} shows the convergence of left-hand side.
		While on right-hand side, $\sup_{0<\varepsilon\leq1}b_{\varepsilon,\tau,a}\to0$,
		$m_{i,\varepsilon}\to0$ as $\tau\to0$, $\varepsilon\to0$. Since $\chi$ was arbitrary, we obtain on
		every $Z_a$
		\begin{equation}\label{eq:supersolution}
			(P_a+\ddc U_a)^n\le e^{U_a}f_a\eta_a^n.
		\end{equation}
		The comparison principle gives
		\[
		u_a\le U_a.
		\]
		Since $a$ is arbitrary we get $u\leq U$ on $q^{-1}(B)$.
		
		\medskip
		\noindent\textbf{Step 3: verify the submean inequality.}
		Let $a\in \Delta^\circ$, for every $z_0\in q^{-1}(a)$ let $\Omega_z\times B_t$ be a product coordinate chart near $z_0$. 
		Let $\sigma$ be a holomorphic map from a neighborhood of $\overline{B}\subset\Delta^\circ$ into this chart with $\sigma(0)=z_0$,
		$q\circ\sigma=\operatorname{id}$.
		Choose local
		potentials $A,R$ of $P,\eta$ on this chart respectively, and set $A_\varepsilon=A+\varepsilon R$.
		We denote the Perron envelopes with boundary value $u_{\varepsilon}|_{q^{-1}(\partial B)}$ by
		$U_\varepsilon$ and $\hat{u}:=\lim_{\varepsilon\to0}U_{\varepsilon}$.
		We will use the crucial properties: $U_{\varepsilon}=u_{\varepsilon}$ on $q^{-1}(\partial B)$ and $U_{\varepsilon}\geq u_{\varepsilon}$ on $q^{-1}(B)$.
		Submean inequality for $U_\varepsilon+A_\varepsilon$ gives
		\[
		(U_\varepsilon+A_\varepsilon)(z_0)
		\le\int_{\partial B}
		(u_\varepsilon+A_\varepsilon)(\sigma(t))\,d\nu(t).
		\]
		Here $d\nu$ is the measure $\frac{d\theta}{2\pi}$ on $\partial B=S^1$.
		Using $u_\varepsilon\le U_\varepsilon$ we know $u\leq \hat{u}$ and conclude that
		\begin{align*}
			(u+A)(z_0)
			\le (\hat{u}+A)(z_0)
			&\leq \int_{\partial B}(\hat{u}+A)(\sigma(t))d\nu(t)\\
			&= \lim_{\varepsilon\to0}\int_{\partial B}(U_{\varepsilon}+A_{\varepsilon})(\sigma(t))d\nu(t)\\
			&=\lim_{\varepsilon\to0}\int_{\partial B}(u_{\varepsilon}+A_{\varepsilon})(\sigma(t))d\nu(t)
			=\int_{\partial B}(u+A_0)(\sigma(t))\,d\nu(t).
		\end{align*}
		
		The function $u$ is upper semicontinuous, since it is the decreasing limit of
		the continuous functions $u_\varepsilon$ from Lemma~\ref{lem:C0}.
		Each non-vertical complex line can be
		parametrized by a map $\sigma$ as above, so submean inequality for $u$ holds on such line. On a
		vertical complex line, submean inequality directly follows from
		$u_t\in\PSH(Z_t,P_t)$. This proves $u+A$ is psh and the proof is complete.
	\end{proof}

	To prove plurisubharmonic variation Theorem~\ref{thm:variation}, we need one more lemma.
	\begin{lemma}\label{lem:missing}
		Let $p:X\to\Delta$ be a proper holomorphic submersion, $a$ a smooth
		closed real $(1,1)$-form with big fiber classes, and $dV_t$ a smooth
		positive family of fiber volume forms. Suppose $\Phi\in\PSH(X,a)$ and
		for $t\ne0$ its restrictions solve Monge--Amp\`ere equation
		\[
		\np{(a_t+\ddc\Phi_t)^n}=e^{\Phi_t}dV_t
		\]
		with full mass. If
		$\vol([a_t])\to\vol([a_0])=:V_0>0$, then $\Phi_0$ is the full-mass
		solution on $X_0$.
	\end{lemma}
	
	\begin{proof}
		Choose a disk $\Delta'\Subset\Delta$ containing $0$. The function $\Phi$
		is bounded above on $p^{-1}(\overline{\Delta'})$. Fatou's lemma then implies
		\begin{equation}\label{eq:fatou}
			V_0=\lim_{t\to0}\vol([\alpha_t])=\lim_{t\to0}\int_{X_t}e^{\Phi_t}dV_t\leq \int_{X_0}(\limsup_{t\to0}e^{\Phi_t})dV_0 \le\int_{X_0}e^{\Phi_0}dV_0.
		\end{equation}
		Thus $\Phi_0\not\equiv-\infty$, and it is $a_0$-psh.
		In holomorphic product coordinate $\Omega_z\times V_t$ we write $a=\ddc A$ and
		$v(z,t)=\Phi(z,t)+A(z,t)$. The fiber equation is
		\[
		\np{(\ddc_zv_t)^n}=e^{v_t}F(z,t)d\lambda_z,
		\qquad F>0\text{ smooth}.
		\]
		We claim that 
		\[\np{(dd^c_zv_0)^n}\geq e^{v_0}F(z,0)d\lambda_z.\]
		To prove the claim, we take $C>0$ so that $(dd^c|z|^2)^n\geq e^{|z|^2-C}F_td\lambda_z$, $\forall t$, and first prove that
		\[(dd^c_z\max\{v_t,|z|^2-C\})^n\geq e^{\max\{v_t,|z|^2-C\}}F_td\lambda_z, \qquad \forall t\neq0.\]
		Indeed, for every $t\neq0$ there is $\varepsilon_j\to0$ so that $\{v_t=|z|^2-C-\varepsilon_j\}$ has zero Lebesgue measure, therefore local property of non-pluripolar products
		in the plurifine topology implies the desired inequality holds for $C$ replaced by $C-\varepsilon_j$. Then letting $\varepsilon_j\to0$ we are done. 
		
		Let $V_C(z,t):=\max\{v(z,t),|z|^2-C\}$ and $V_{r,C}(z):=\sup_{|t|\leq r}V_C(z,t)$ for $r>0$.
		It is clear that $V_{r,C}$ is a bounded psh function.
		Choquet lemma gives a sequence $v_j:=\max\{v_{t_j},|z|^2-C\}$, $|t_j|\leq r$ so that $V_{r,C}=(\sup v_j)^*$.
		Let $F_r(z):=\min_{|t|\leq r}F(t,z)$, then $(dd^cv_j)^n\geq e^{v_j}F_rd\lambda_z$ for every $j$. 
		Local property of Monge--Amp\`ere operator yields $(dd^c_z\max\{v_1,\cdots,v_j\})^n\geq e^{\max\{v_1,\cdots,v_j\}}F_rd\lambda_z$ for every $j$ as above. 
		Taking the limit, Bedford--Taylor convergence theorem implies
		\[(dd^cV_{r,C})^n\geq e^{V_{r,C}}F_rd\lambda_z.\] 
		Finally $V_{r,C}(z)\searrow V_C(z,0)=\max\{v(z,0),|z|^2-C\}$ as $r\to0$, because for each $z$, function $t\to V_C(z,t)$ is subharmonic hence strongly upper semicontinuous.
		Bedford--Taylor convergence theorem implies
		\[(dd^c_z\max\{v(z,0),|z|^2-C\})^n\geq e^{\max\{v(z,0),|z|^2-C\}}F(z,0)d\lambda_z.\]
		Let $C\to\infty$, the claim is proved.
		As a result,
		\[
		\np{(a_0+\ddc\Phi_0)^n}\ge e^{\Phi_0}dV_0.
		\]
		Together with \eqref{eq:fatou} we obtain that
		\[
		V_0\le\int_{X_0}e^{\Phi_0}dV_0
		\le\int_{X_0}\np{(a_0+\ddc\Phi_0)^n}
		\le\vol([a_0])=V_0.
		\]
		All inequalities are equalities. This completes the proof.
	\end{proof}
	
	\begin{proof}[Proof of Theorem~\ref{thm:variation}]
		Apply Theorem~\ref{prop:model} and work
		first over $\Delta^\circ=\Delta\setminus T$. 
		Recall that $\eta$ is a K\"ahler form on $Z$, $\beta$ is a K\"ahler form on $X$. For some smooth function $h$ on $Z$ it holds that
		\begin{equation}
			\theta_\eta+\mu^*\beta=P+\Theta_D+dd^ch, \qquad f=e^{-h}\Pi_i|s_i|^{2d_i}.
		\end{equation}
		Let $u$ be the $P$-plurisubharmonic function so that
		\begin{align*}
			(P_t+dd^cu_t)^n=e^{u_t}f_t\eta_t^n,\qquad t\in \Delta^{\circ}.
		\end{align*}
		Define
		\[
		\psi_Z:=u+\log|s_D|^2-h.
		\]
		Then, as currents on $Z$,
		\begin{equation}\label{eq:targetZ}
			\theta_\eta+\mu^*\beta+\ddc\psi_Z=P+\ddc u+[D].
		\end{equation}
		Let
		\[
		J:=\log\frac{\mu^*\omega_t^n}{\eta_t^n}.
		\]
		The Poincar\'e--Lelong identity is
		\[
		\ddc J=\mu^*\theta_\omega-\theta_\eta+[E_\mu],
		\]
		where $E_\mu=K_{Z/\Delta}-\mu^*K_{X/\Delta}$.
		The divisorial terms of $\psi_Z$ and $J$ are given by $D=E_{\alpha}+E_{\mu}$ and $E_\mu$ respectively,
		hence locally
		$\psi_Z-J=u+\log|s_{E_\alpha}|^2+\text{a smooth function}$.
		Combining this with \eqref{eq:targetZ} gives
		\begin{equation}\label{eq:positivepullback}
			\mu^*\alpha+\ddc(\psi_Z-J)
			=P+\ddc u+[E_\alpha]\ge0.
		\end{equation}
		We take the non-pluripolar product of both sides on each fiber $Z_t$ and get that
		\begin{equation}\label{eq:Z MA}
			\left\langle\big((\mu^*\alpha)_t+dd^c(\psi_{Z,t}-J_t)\big)^n\right\rangle
			=\langle(P_t+dd^cu_t+[E_{\alpha,t}])^n\rangle=(P_t+dd^cu_t)^n=e^{\psi_{Z,t}}\eta_t^n.
		\end{equation}
		$E_{\alpha,t}$ is pluripolar so we can omit it in the non-pluripolar product.
		On the isomorphism locus of $\mu:Z\to X$, $\mu^*\alpha$-psh function $\psi_Z-J$ induces an $\alpha$-psh function $\widetilde\varphi$. 
		It extends across
		the center of modification, because $\psi_Z-J$ is quasi-psh and hence is bounded from above on
		any inverse image of compact set, the standard removable-singularity theorem for psh
		functions applies to $\widetilde{\varphi}$.

		We claim that for $t\in\Delta^{\circ}$,
		\begin{equation}\label{eq:X MA}
			\np{(\alpha_t+\ddc\widetilde\varphi_t)^n}
			=e^{\widetilde\varphi_t}\omega_t^n\qquad \text{on $X_t$}.
		\end{equation}
		In fact by Proposition~\ref{prop:model}, $\mu_t:Z_t\to X_t$ is a modification, thus the pull-back of the left-hand side coincides with that of \eqref{eq:Z MA}.
		On the other hand, $\mu_t^*(e^{\widetilde\varphi_t}\omega_t^n)=e^{\psi_{Z,t}-J_t}\mu^*\omega_t^n=e^{\psi_{Z,t}}\eta_t^n$ follows from the definition of $J$. This proves our claim.
		The total mass of left-hand side of \eqref{eq:X MA} equals $\int_{Z_t}P_t^n$.
		By Proposition~\ref{prop:model}, this is $\vol([\alpha_t])$. Hence by 
		\cite[Theorem 6.1]{BEGZ} $\widetilde\varphi_t$ is the unique solution with minimal singularities on every good fiber.
		
		Eventually we deal with discrete fibers over $T$.
		Fix $t_0\in T$, the good-fiber solutions have a uniform upper bound near $t_0$. Indeed $\alpha_t\le C\omega_t$ for some uniform $C$, and for $t$ near $t_0$, the normalized functions $\varphi_t-\sup_{X_t}\varphi_t$ satisfies
		\[\int_{X_t}|\varphi_t-\sup_{X_t}\varphi_t|\omega_t^n\leq C'\]
		for some uniform $C'>0$ by
		\cite[Proposition 3.3]{DNGG}. Jensen's inequality and
		\[
		\int_{X_t}e^{\varphi_t}\omega_t^n=\vol([\alpha_t])=V>0
		\]
		give the required uniform upper bound of $\sup_{X_t}\varphi_t$ as in Lemma~\ref{lem:C0}. Thus $\widetilde\varphi$ has a unique $\alpha$-psh
		extension $\Phi$ across the fiber $X_t$. 
		Volume constancy of $\vol([\alpha_t])$ and
		Lemma~\ref{lem:missing} show that $\Phi_t$ solves the corresponding equation on $X_t$. This finishes the proof.
	\end{proof}

	\section{Construct relative canonical metric with minimal singularities}
	\label{sec:canonicalmetric}
	
	\begin{theorem}\label{thm:metric}
		Let $p:X\rightarrow\Delta$ be as before, $\omega$ be
		a K\"ahler form on $X$ and let $\theta=\theta_{\omega}$ be the curvature form of induced metric on $K_{X/\Delta}$. After shrinking $\Delta$, there exists
		$\Psi\in\PSH(X,\theta)$ such that for every $t\in\Delta$,
		the restriction $\Psi_t$ is not identically $-\infty$ and has minimal
		singularities in $\PSH(X_t,\theta_{t})$. 
	\end{theorem}
	
	\begin{proof}
		For $0<\eta\le1$, put
		\[
		\alpha_\eta=\theta+\eta\omega.
		\]
		Theorem~\ref{thm:variation} applies to give a
		potential $\varphi_\eta\in\PSH(X,\alpha_\eta)$ satisfying for each $t\in\Delta$,
		\[
		\np{(\alpha_{\eta,t}+\ddc\varphi_{\eta,t})^n}
		=e^{\varphi_{\eta,t}}\omega_t^n
		\]
		with $\varphi_{\eta,t}$ has full Monge-Amp\`ere mass. For each fixed $\eta$, Theorem~\ref{prop:model}
		shows that
		\[
		v_\eta:=\vol([\alpha_{\eta,t}])>0
		\]
		is constant on $\Delta$. Closedness
		of $\omega$ proves
		$W:=\int_{X_t}\omega_t^n$ is constant on $\Delta$.
		Define
		\begin{equation}\label{eq:massnormalizedKE}
			u_\eta=\varphi_\eta-\log(v_\eta/W),\qquad
			d\nu_t=W^{-1}\omega_t^n.
		\end{equation}
		Then
		$u_\eta\in\PSH(X,\alpha_\eta)$ and
		\begin{equation}\label{eq:probabilityKE}
			\frac{\np{(\alpha_{\eta,t}+\ddc u_{\eta,t})^n}}{v_\eta}
			=e^{u_{\eta,t}}d\nu_t,\qquad
			\int_{X_t}e^{u_{\eta,t}}d\nu_t=1.
		\end{equation}
		
		Fix a relatively compact disk $\Delta'\Subset\Delta$.
		There is $A>0$ such that $\alpha_\eta\le A\omega$ on
		$p^{-1}(\overline{\Delta'})$ for all $0<\eta\le1$.
		The uniform estimate
		\cite[Proposition 3.3]{DNGG}, applied to the smooth family and to
		$A\omega$, gives
		\[
		M_{\eta,t}-C\le\int_{X_t}u_{\eta,t}d\nu_t\le M_{\eta,t},
		\qquad M_{\eta,t}=\sup_{X_t}u_{\eta,t}.
		\]
		The normalization $\int_{X_t}e^{u_{\eta,t}}d\nu_t=1$ in \eqref{eq:probabilityKE} and $\int_{X_t}d\nu_t=1$ imply
		$M_{\eta,t}\ge0$. Jensen's inequality gives
		\[
		0=\log\int_{X_t}e^{u_{\eta,t}}d\nu_t
		\ge\int_{X_t}u_{\eta,t}d\nu_t
		\ge M_{\eta,t}-C.
		\]
		Thus
		\begin{equation}\label{eq:uniformuppersup}
			0\le M_{\eta,t}\le C
			\quad(t\in\Delta',\ 0<\eta\le1).
		\end{equation}
		In particular, the functions $u_\eta$ are locally uniformly bounded
		above on $X$.
		
		We fix one fiber $X_t$. The normalized potential
		$\widehat u_{\eta,t}=u_{\eta,t}-M_{\eta,t}$ has minimal
		singularities by \cite[Theorem 6.1]{BEGZ}. Subtracting its supremum
		does not change its Monge--Amp\`ere measure. Its normalized equation
		has density $e^{u_{\eta,t}}\le e^C$ with respect to $d\nu_t$.
		Theorem~\ref{thm:uniform} consequently gives a constant $C_t$,
		independent of $\eta$, such that
		\begin{equation}\label{eq:fiberlower}
			u_{\eta,t}\ge\widehat u_{\eta,t}
			\ge V_{\theta_t+\eta\omega_t}-C_t
			\ge V_{\theta_t}-C_t.
		\end{equation}
		The last inequality holds because $[\theta_t+\eta\omega_t]\geq[\theta_t]\geq0$.
		
		Choose any sequence $\eta_j\downarrow0$ and set
		\begin{equation}\label{eq:tailenvelope}
			U_N=\left(\sup_{j\ge N}u_{\eta_j}\right)^*,\qquad
			\Psi=\lim_{N\to\infty}U_N.
		\end{equation}
		The upper-semicontinuous regularization here is taken on the total
		space $X$. Since $M_{\eta,t}=\sup u_{\eta,t}$ is uniformly (with respect to $\eta,t$) bounded from above, $U_N\neq+\infty$ is $(\theta+\eta\omega)$-psh. The sequence $U_N$
		decreases, and \eqref{eq:fiberlower} implies on each fiber,
		\begin{equation}\label{eq:tailfiberbound}
			U_N|_{X_t}\ge V_{\theta_t}-C_t.
		\end{equation}
		The uniform lower bound $M_{\eta,t}\geq0$ also shows the decreasing limit $\Psi$ is not identically $-\infty$.
		Thus $\Psi\in \operatorname{PSH}(X,\theta)$ and \eqref{eq:tailfiberbound} implies $\Psi_t\ge V_{\theta_t}-C_t$. 
		The proof is therefore complete.
	\end{proof}
	\begin{remark}
		It is clear that $\Phi$ is independent of $s\in H^0(X_0,mK_{X_0})$.
	\end{remark}

	\section{Extension of pluricanonical sections}
	\label{sec:pluricanonicalextension}
	\begin{theorem}[=Theorem~\ref{thm:plurigenera}]\label{thm: inv-Kah}
		Let $p:X\rightarrow \Delta$ be a proper holomorphic submersion from a K\"ahler manifold $(X,\omega)$ to the unit disk. Then any $s\in H^0(X_0,mK_{X_0})$ can be extended to a section $\hat s\in H^0(X, mK_X)$. In particular, the complex dimension of $H^0(X_t,mK_{X_t})$ is invariant with $t\in \Delta$.
	\end{theorem}

	\begin{proof} 
		As explained in the introduction we may assume $K_{X_t}$ is pseudoeffective for every $t\in\Delta$, cf. \cite{CP25}.
		Let $\theta_{m-1}:=(m-1)\theta\in c_1((m-1)K_X)$ be the curvature form associated to the metric $h^{m-1}_\omega$ induced from a K\"ahler metric $\omega$ on $X$. Now $i\partial\bar\partial \left(\frac{m-1}{m}\log|s|^2\right)=\theta_{m-1}|_{X_0}+i\partial\bar\partial\varphi$ is a closed positive current on $X_0$ such that  $\varphi$ has analytic singularities. From Theorem~\ref{thm:metric}, there is an $\theta_{m-1}$-psh function $\Phi_{m-1}=(m-1)\Phi$ on $X$ after shrinking $\Delta$, such that $\mathcal I(\varphi)\subset \mathcal I(\Phi_{m-1}|_{X_0})$. Write $mK_X=K_X+(m-1)K_X:=K_X+L$.   Equipped with $L$ the metric $h^{m-1}_\omega e^{-\Phi_{m-1}}$. Since $X$ is a weakly pseudoconvex K\"ahler manifold, from the general $L^2$ extension Theorem \ref{thm: ext}, we get an extension $\hat s\in H^0(X, mK_{X})$ of $s\in H^0(X_0, mK_{X_0})$. 
		
		\noindent This completes the proof of the main theorem.
	\end{proof}
	
	\noindent\textbf{AI Disclosure.}
	After getting our idea about the connection with bimeromorphic model, ChatGPT suggested Hacon-Xie's paper \cite{HX}, which is what we want.

\end{document}